%% file: ms.tex
\documentclass[
	fontsize=11pt,
	parskip=half,
]{scrartcl}

\input{header.tex}
\input{formatting.tex}

\title{The Decoupled Haydys-Witten Equations\\ and a Weitzenböck Formula}
\date{\today}

\begin{document}

\input{title}
\input{abstract}

\input{introduction}
\input{acknowledgements}
\input{body}

\printbibliography[heading=bibintoc]

\end{document}

%% file: header.tex
\usepackage[dvipsnames,hyperref]{xcolor}
			\definecolor{dark-gray}{gray}{0.1}
			\definecolor{dark-blue}{RGB}{0,0,80}
			\definecolor{light-gray}{gray}{0.9}

\usepackage[english]{babel}
\usepackage[autostyle]{csquotes}			
			\MakeOuterQuote{"}

\usepackage{amsmath} 					
\usepackage{mathtools}				
			\mathtoolsset{showonlyrefs,mathic}
\usepackage{amssymb}					
\usepackage{aliascnt}					
\usepackage{amsthm} 					
			\theoremstyle{definition}
				\newtheorem{definition}{Definition}[section]
				
				\newtheorem*{definition*}{Definition}
			\theoremstyle{plain}
				\newaliascnt{theorem}{definition} 								
				\newtheorem{theorem}[theorem]{Theorem}						
				\aliascntresetthe{theorem} 												
				\newtheorem*{theorem*}{Theorem}										
				\newaliascnt{proposition}{definition}
				\newtheorem{proposition}[proposition]{Proposition}
				\aliascntresetthe{proposition}
				
				\newaliascnt{corollary}{definition}
				\newtheorem{corollary}[corollary]{Corollary}
				\aliascntresetthe{corollary}
				
				\newtheorem*{corollary*}{Corollary}
				\newaliascnt{lemma}{definition}
				\newtheorem{lemma}[lemma]{Lemma}
				\aliascntresetthe{lemma}
				
				\newtheorem*{lemma*}{Lemma}
				\newtheorem{bigtheorem}{Theorem}

			\theoremstyle{remark}
				\newtheorem*{remark}{Remark}
				
				\newtheorem*{example}{Example}

\usepackage{array}						
\usepackage{multicol}
\usepackage[inline]{enumitem}					

\usepackage[
	backend=biber,
	bibstyle=authoryear,
	citestyle=alphabetic,
	hyperref,
	sorting=ynt,
	giveninits=true,
	maxbibnames=99,
	backref=false,
	doi=false
	]{biblatex}
			\renewbibmacro{in:}{%
				\ifentrytype{article}{}{\printtext{\bibstring{in}\intitlepunct}}
			}
			\DefineBibliographyStrings{english}{backrefpage = {see p.}, backrefpages = {see pp.}}
			\renewbibmacro*{doi+eprint+url}{
				\iftoggle{bbx:doi}{\printfield{doi}}{}
				\iftoggle{bbx:eprint}{\usebibmacro{eprint}}{}
				\iftoggle{bbx:url}{
					\iffieldundef{doi}{
						\iffieldundef{eprint}{
							\usebibmacro{url+urldate}
						}{}
					}{}
				}
			}
			\AtEveryBibitem{\ifentrytype{article}{\clearfield{issn}}{}}
			\DeclareFieldFormat{labelalphawidth}{\mkbibbrackets{#1}}

			\defbibenvironment{bibliography}
				{\list
					 {\printtext[labelalphawidth]{%
							\printfield{labelprefix}%
					\printfield{labelalpha}%
							\printfield{extraalpha}}}
					 {\setlength{\labelwidth}{\labelalphawidth}%
						\setlength{\leftmargin}{\labelwidth}%
						\setlength{\labelsep}{\biblabelsep}%
						\addtolength{\leftmargin}{\labelsep}%
						\setlength{\itemsep}{\bibitemsep}%
						\setlength{\parsep}{\bibparsep}}%
						}
				{\endlist}
				{\item}

\usepackage[
	pdfusetitle,
	colorlinks=true, pdfborder={0 0 0},
	citecolor={dark-blue}, urlcolor={dark-blue},	linkcolor={dark-blue}
	]{hyperref}
\usepackage[numbered, open, openlevel=0]{bookmark}

\newcommand{\thalf}{\tfrac{1}{2}}

\newcommand{\del}{\partial}
\DeclareMathOperator{\Tr}{Tr}

\newcommand{\abs}[1]{\left| #1 \right|}
\newcommand{\norm}[1]{\left\lVert #1 \right\rVert}

\newcommand{\set}[1]{\left\{\ #1 \ \right\} } 							
\newcommand{\suchthat}{\mathrel{}\middle|\mathrel{}}		

\newcommand{\spans}[1]{\operatorname{span}\left\{ #1 \right\}}		

\newcommand{\hodge}{\star}

\newcommand\Restrict[2]{{\left. \kern-\nulldelimiterspace #1 \right\vert_{#2} }}								 
\newcommand\restrict[1]{\raisebox{-.2ex}{\ensuremath \vert}_{#1}}  

\newcommand{\N}{\mathbb{N}}			
\newcommand{\R}{\mathbb{R}} 		
\newcommand{\C}{\mathbb{C}}			

\DeclareMathOperator{\ad}{ad}

\DeclareMathOperator{\Hom}{Hom}

\newcommand{\HW}[1][]{{\mathop{\mathbf{HW}}{\hspace{-0.2em}}_{#1\, \/}}}
\newcommand{\KW}[1][]{{\mathop{\mathbf{KW}}{\hspace{-0.2em}}_{#1\, \/}}}
\newcommand{\VW}[1][]{{\mathop{\mathbf{VW}}{\hspace{-0.2em}}_{#1\, \/}}}
\newcommand{\dHW}[1][]{{\mathop{\mathbf{dHW}}{\hspace{-0.2em}}_{#1\, \/}}}

\AtBeginDocument{%
\let\Re\relax
\DeclareMathOperator{\Re}{Re}

}

%% file: formatting.tex
\usepackage[left=2.5cm, right=2.5cm, top=3cm, bottom=3cm]{geometry}
\usepackage{setspace}
\let\savemathfrak\mathfrak											
\usepackage[ttscale=.85]{libertine} 
\usepackage[T1]{fontenc}
\usepackage[amsthm,upint]{libertinust1math} 
\let\mathfrak\savemathfrak											
\DeclareMathAlphabet\mathcal{OMS}{cmsy}{m}{n}		

\addtokomafont{paragraph}{\scshape\mdseries}
\addtokomafont{section}{\large\scshape}

%% file: title.tex
\makeatletter
\renewcommand*{\thefootnote}{\fnsymbol{footnote}}

\thispagestyle{empty}

{

\setstretch{1.3}
\Large\bfseries\scshape
\@title

}

\vspace{1em}

Michael Bleher\footnotemark[2] \\
{\small\itshape
Institute for Mathematics, Heidelberg University, Im Neuenheimer Feld 205, Heidelberg, Germany.
}

\footnotetext[2]{mbleher@mathi.uni-heidelberg.de}

\renewcommand*{\thefootnote}{\arabic{footnote}}
\makeatother

%% file: abstract.tex
\paragraph{Abstract.}
The Haydys-Witten equations are partial differential equations on five-dimensional Riemannian manifolds that are equipped with a non-vanishing vector field $v$.
Conjecturally, their solutions determine the Floer differential in a gauge-theoretic approach to Khovanov homology.
This article introduces a certain decoupled version of the Haydys-Witten equations, a specialization of the Haydys-Witten equations that exhibits a Hermitian Yang-Mills structure.
These equations exist whenever the vector bundle defined by the orthogonal complement of $v$ admits an almost Hermitian structure.
We investigate the relation between the full Haydys-Witten equations and their decoupled version on manifolds with poly-cylindrical ends and boundaries, and find conditions under which the Haydys-Witten equations reduce to the decoupled equations.
This relies on a Weitzenböck-like formula that shows that the difference between the full Haydys-Witten equations and the decoupled equations is governed by the asymptotic behaviour of solutions near boundaries and non-compact ends.
Regarding the analysis near boundaries, we provide a detailed analysis of the polyhomogeneous expansion of Haydys-Witten solutions with twisted Nahm pole boundary conditions, generalizing work of Siqi He in the untwisted case. 
The corresponding analysis at non-compact ends relies on a vanishing theorem by Nagy-Oliveira that was recently generalized by the author.

%% file: introduction.tex
\section{Introduction}
Let $(M^5, g)$ be a five-manifold with poly-cylindrical ends, where ends may be located at either finite or infinite geodesic distance.
Assume $M^5$ admits a non-vanishing unit vector field $v\/$ and that the subbundle $\ker g(v,\cdot) \subset TM^5 $ admits an almost Hermitian structure.
This means that there is an almost complex structure $J:\ker g(v,\cdot) \to \ker g(v,\cdot)$ that is compatible with the metric, i.e. $g(J\cdot, J\cdot) = g(\cdot, \cdot)$.

In this article we investigate the Haydys-Witten equations on manifolds $(M^5, g, v, J)$.
The existence of $J$ provides a specialization of the equations that we call \emph{decoupled} Haydys-Witten equations.
Crucially, the decoupled version of the equations exhibits a Hermitian Yang-Mills structure, which provides additional tools in solving the equations.
This structure becomes most apparent in the $4\mathcal{D}$-formulation of the Haydys-Witten equations, an extension of Witten's $3\mathcal{D}$-formulation of the extended Bogomolny equations (EBE)~\cite{Witten2011}, which is introduced below and used throughout the introduction.
The main contribution of this article consists in working out conditions under which the Haydys-Witten equations reduce to the decoupled version.

Curiously, in the context of Witten's gauge theoretic approach to homological knot invariants~\cite{Witten2011}, manifolds are generally equipped with the additional structure $(g,v,J)$.
In that situation one considers the Haydys-Witten equations on five-manifolds of the form $M^5= \R_s \times X^3 \times \R_y^+$, equipped with a product metric $g$, and sets $v = \del_y$.
The subbundle $\ker g(v,\cdot)$ is then simply the tangent space of $\R_s \times X^3$, which always admits an almost complex structure, as it is an open and orientable four-manifold.
Conjecturally, when $X^3=\R^3$ or $S^3$ and in the presence of a magnetically charged knot $K \subset X^3$, the homology groups obtained from $\theta$-Kapustin-Witten solutions (stationary Haydys-Witten solutions) at $s\to\pm \infty$ modulo Haydys-Witten instantons reproduces a knot invariant known as Khovanov homology.
Hence, the results presented here may offer a fresh perspective on the gauge theoretic approach to Khovanov homology.

In the following we briefly introduce the $4\mathcal{D}$-formulation of the Haydys-Witten equations and the decoupled version of the equations, see \autoref{sec:vanishings-weitzenböck} for a more detailed, global discussion.
For this, let $G$ denote a compact Lie group, $G_{\mathbb{C}}$ its complexification, $E\to M^5$ a $G$-principal bundle, and $E_{\mathbb{C}}$ the associated $G_{\mathbb{C}}$-principal bundle.
Furthermore, let $\mathcal{A}(E)$ denote gauge connections and write $\Omega^2_{v,+}(M^5)$ for Haydys' self-dual two-forms with respect to $v$~\cite{Haydys2010}.

Given a pair $(A,B) \in \mathcal{A}(E)\times \Omega^2_{v,+}(M^5,\ad E)$ and an almost complex structure $J$, one can locally define four differential operators $\mathcal{D}_\mu$ that act on sections of $\ad E_{\mathbb{C}}$.
To that end, consider normal coordinates $(x^i,y)_{i=0,1,2,3}$ near a point $p$, chosen in such a way that $v = \del_y$ and that $J$ takes the canonical form with respect to the coordinate vector fields $\del_i$ at $p$.
In these coordinates, $B = \sum_{a,b,c = 1}^3 \phi_a (dx^0\wedge dx^a + \thalf \epsilon_{abc} dx^b\wedge dx^c)$.
The four differential operators are defined by
\begin{align}
	\mathcal{D}_0 &= \nabla^A_0 + i\nabla^A_1 &
	\mathcal{D}_1 &= \nabla^A_2 + i\nabla^A_3 \\
	\mathcal{D}_2 &= \nabla^A_y - i [\phi_1, \cdot] &
	\mathcal{D}_3 &= [\phi_2, \cdot] - i [\phi_3,\cdot]
\end{align}
There is a complex conjugation, induced from $\ad E_{\mathbb{C}}$, that we denote by $\overline{\mathcal{D}_\mu}$.
Furthermore, $G_{\mathbb{C}}$-valued gauge transformations act on the operators by conjugation, i.e. $\mathcal{D}_\mu \mapsto g \mathcal{D}_\mu g^{-1}$.

In this formulation, the Haydys-Witten equations $\HW[v](A,B) = 0$ are given by
\begin{align}
	[\overline{\mathcal{D}_0} , \overline{\mathcal{D}_i} ] - \thalf \varepsilon_{ijk} [\mathcal{D}_j, \mathcal{D}_k] &= 0 \ , \qquad i = 1,\ldots,3\ , \label{eq:vanishings-HW-4Ds-gauged-Nahm}\\
	\sum_{\mu = 0}^3 [\mathcal{D}_\mu, \overline{\mathcal{D}_\mu}] &= 0\ . \label{eq:vanishings-HW-4Ds-moment-map-condition}
\end{align}
A typical approach in solving equations of this type is to utilize their symmetry properties.
Since there is an action by complex gauge transformation, one natural idea is to use a Donaldson-Uhlenbeck-Yau type approach, where one first extracts some underlying holomorphic data from $G_{\C}$-invariant parts of the equations, and subsequently hopes to find a complex gauge transformation that ensures also the remaining equations are satisfied.

Unfortunately, although the Haydys-Witten equations are invariant under $G$-valued gauge transformations and the action lifts naturally to $G_{\C}$, neither~\eqref{eq:vanishings-HW-4Ds-gauged-Nahm} nor~\eqref{eq:vanishings-HW-4Ds-moment-map-condition} are invariant under $G_{\C}$-valued gauge transformations.
There is, however, a subset of solutions for which the three equations in~\eqref{eq:vanishings-HW-4Ds-gauged-Nahm} decompose into their $G_{\C}$-invariant parts.
This is given by solutions that satisfy the following equations:
\begin{align}\label{eq:vanishings-dHW-4D}
\begin{split}
	[ \mathcal{D}_\mu, \mathcal{D}_\nu ] &= 0 \ , \qquad \mu,\nu = 0, \ldots, 3\ , \\
	\sum_{\mu = 0}^3 [\mathcal{D}_\mu, \overline{\mathcal{D}_\mu}] &= 0\ .
\end{split}
\end{align}
We refer to~\eqref{eq:vanishings-dHW-4D} as \emph{decoupled Haydys-Witten equations} and denote them by ${\dHW[v, J](A,B)=0}$.
A global version of these equations is provided in \autoref{sec:vanishings-weitzenböck}.
The commutativity equations are $G_{\mathbb{C}}$-invariant and can be interpreted as a complex moment map condition in a hyper-Kähler reduction, while the remaining equation provides the real moment map condition.
Put differently, the decoupled equations exhibit a Hermitian Yang-Mills structure, such that a Donaldson-Uhlenbeck-Yau type approach and other powerful tools become available.

Clearly, whenever $\dHW[v,J](A,B) = 0$, then $\HW[v](A,B) = 0$.
Here, we prove that in certain situations the reverse is true, such that the Haydys-Witten equations reduce to the decoupled Haydys-Witten equations.
This is controlled, on the one hand by the (asymptotic) geometry, and on the other hand by the asymptotic behaviour of the fields $(A,B)$, at boundaries and cylindrical ends of $M^5$.
Crucially, we need to assume that
\begin{enumerate*}[label=(\alph*)]
\item boundaries of $M^5$ are flat and $(A,B)$ satisfies Nahm pole boundary conditions with knot singularities, and
\item non-compact ends of $M^5$ are asymptotically locally Euclidean (ALE) or flat (ALF) gravitational instantons and $(A,B)$ approaches a finite energy solution of the $\theta$-Kapustin-Witten equations.
\end{enumerate*}
For the sake of brevity, we omit further technical conditions for now and instead refer to assumptions (A1)~-~(A4) in \autoref{sec:vanishings-proof-main-result} for details.
With that understood, our main result is as follows (this is \autoref{thm:vanishings-decoupling}).
\begin{bigtheorem}\label{bigthm:decoupling}
Let $G=SU(2)$, $M^5$ a manifold with poly-cylindrical ends, $v$ a non-vanishing vector field that approaches ends at a constant angle, and $J$ an almost Hermitian structure on $\ker g(v,\cdot)$.
Assume $\HW[v](A,B)=0$ and that assumptions \textnormal{(A1)}~-~\textnormal{(A4)} are satisfied, then $\dHW[v,J](A,B)=0$.
\end{bigtheorem}

The proof of \autoref{bigthm:decoupling} is based on a Weitzenböck formula of the form
\begin{align} \label{eq:vanishings-intro-weitzenböck}
	\int_{M^5} \norm{\HW[v](A,B)}^2 = \int_{M^5} \norm{\dHW[v,J](A,B)}^2 + \int_{M^5} d \chi
\end{align}
From this it's clear that whenever $\HW[v](A,B) = 0$ and any boundary contributions (including contributions from non-compact ends) in $\int_{M^5} d\chi$ vanish, then also $\dHW[v,J](A,B) = 0$.
The key insights of this article lie in determining conditions under which all boundary contributions vanish, if one imposes Nahm pole boundary conditions at finite distances and asymptotically stationary solutions at infinity.
Due to the two different flavours of boundary conditions, we rely on two distinct facts:
Elliptic regularity of the Nahm pole boundary conditions, on the one hand, and a vanishing theorem for solutions of $\theta$-Kapustin-Witten solutions on ALE and ALF spaces, on the other.
Let us shortly explain how these two facts appear in the proof.

First, the Nahm pole boundary conditions for $(A,B)$ state, in particular, that at order $y^{-1}$ the fields satisfy the extended Bogomolny equations (EBE), which in the $4\mathcal{D}$ formalism correspond to
\begin{align}
	\mathcal{D}_0 &= 0\ , &
	[ \mathcal{D}_i, \mathcal{D}_j ] &= 0 \qquad i,j = 1, \ldots, 3\ , &
	\sum_{i = 1}^3 [\mathcal{D}_i, \overline{\mathcal{D}_i}] &= 0\ .	
\end{align}
This means that the leading order terms are already solutions of the decoupled Haydys-Witten equations.
As will be discussed in detail below, elliptic regularity of the Haydys-Witten equations states that deviations from the EBE-solutions can only appear at order $y^{1+\delta}$, for some $\delta>0$~\cite{Mazzeo2013a, Mazzeo2017, He2018a}.
This, in turn, implies that $\chi$ only involves terms of order $y^\delta$.
As a consequence, any contributions to $\int d \chi$ from boundaries with Nahm pole boundary conditions vanish.

We expect that \autoref{bigthm:decoupling} is also generally true in the presence of knot singularities.
Unfortunately, the twisted knot singularity solutions are only known implicitly (by a continuation argument)~\cite{Dimakis2022}, such that extracting information from elliptic regularity is difficult.
Although we currently have no proof for this extension of the result, we include a discussion of the relevant boundary conditions and state a necessary condition that is known to be satisfied in the untwisted case.

Second, and perhaps more surprisingly, the boundary terms also vanish at asymptotic ends when the fields approach stationary solutions of the Haydys-Witten equations, or equivalently solutions of the $\theta$-Kapustin-Witten equations, with finite energy.
This relies on a vanishing theorem that was originally conjectured by Nagy and Oliveira~\cite{Nagy2021} and for which a proof appeared recently in~\cite{Bleher2023a}.
The vanishing theorem states that for a finite energy solution of the $\theta$-Kapustin-Witten equations on an ALE or ALF gravitational instanton $A$ is flat, $\phi$ is $\nabla^A$-parallel, and $[\phi \wedge \phi] = 0$.

This article is structured as follows.
In \autoref{sec:vanishings-weitzenböck} we summarize Haydys' geometry and the Haydys-Witten equations, define the decoupled Haydys-Witten equations, and establish the promised Weitzenböck formula.
In \autoref{sec:vanishings-setting} we further specify the basic geometric setting that is necessary to specify boundary conditions and which is used in evaluating the integral of the exact term in the Weitzenböck formula.
A key step in this article is an investigation of the polyhomogeneous expansion of a twisted Nahm pole solution of the Haydys-Witten equations, which is presented in \autoref{sec:vanishings-polyhomogeneous-expansion-of-nahm-pole-solutions}.
Subsequently, in \autoref{sec:vanishings-asymptotics-of-chi}, we determine the asymptotic behaviour of $\chi$ at the various boundaries and ends, filling in the details of the remaining boundary conditions as we go.
Finally, in \autoref{sec:vanishings-proof-main-result}, we bring everything together and show that in certain situations the boundary term in~\eqref{eq:vanishings-intro-weitzenböck} vanishes, which immediately implies \autoref{bigthm:decoupling}.

%% file: acknowledgements.tex
\paragraph{Acknowledgements}
This work is funded by the Deutsche Forschungsgemeinschaft (DFG, German Research Foundation) under Germany's Excellence Strategy EXC 2181/1 - 390900948 (the Heidelberg STRUCTURES Excellence Cluster).

%% file: body.tex
\section{The Decoupled Haydys-Witten Equations and a Weitzenböck Formula}
\label{sec:vanishings-weitzenböck}

In this section, we introduce the global version of the Haydys-Witten equations and their decoupled version, and establish a Weitzenböck formula that relates it to the full Haydys-Witten equations.

Let $(M^5, g)$ be a complete Riemannian manifold.
Consider a principal $G$-bundle $E\to M^5$ for some compact Lie group $G$, write $\mathcal{A}(E)$ for the space of gauge connections, and denote by $\ad E$ the adjoint bundle associated to the Lie algebra $\mathfrak{g}$ of $G$.

Assume $M^5$ admits a non-vanishing unit vector field $v\/$ and write ${\eta = g(v,\cdot) \in \Omega^1(M)}$ for its dual one-form.
Since $v$ is non-vanishing, the distribution $\ker \eta \subset TM^5$ is regular, i.e. a vector subbundle, of rank 4.
Observe that the pointwise linear map
\begin{align}
	T_\eta: \Omega^2(M) \to \Omega^2(M), \quad
	\omega \mapsto \hodge_{5} ( \omega \wedge \eta)
\end{align}
has eigenvalues $\{ -1, 0, 1\}$, such that $\Omega^2(M)$ decomposes into the corresponding eigenspaces:
\begin{align}
	\Omega^2(M) = \Omega^2_{v,-}(M) \oplus \Omega^2_{v,0}(M) \oplus \Omega^2_{v,+}(M)
\end{align}
We use the notation $\omega^+$ to denote the part of $\omega$ that lies in $\Omega^2_{v,+}(M)$.

Consider a pair $(A,B) \in \mathcal{A}(E) \times \Omega^2_{v,+}(\ad E)$, consisting of a gauge connection $A$ and a self-dual two form $B$.
Let $\nabla^A$ denote the covariant derivative on $\Omega^2_{v,+}(\ad E)$ induced by $A$ and the Levi-Civita connection, and define $\delta_A^+ : \Omega^2_{v,+}(M,\ad E) \to \Omega^1(M, \ad E)$ by the composition $\delta_A^+ B = - g^{\mu\nu} \iota_\mu \nabla^A_\nu B$.
The Haydys-Witten equations are defined by:
\begin{align}
	\begin{split}
	F_A^+ - \sigma(B, B) - \nabla_{v}^A B &= 0\ , \\	
	\imath_{v}\ F_A -  \delta_A^+ B &= 0\ .
	\end{split}
	\label{eq:vanishings-setting-Haydys-Witten-equations}
\end{align}

Assume now that there is an almost complex structure $J$ on the vector bundle $\ker \eta \to M^5$ that is compatible with the metric, i.e. $g(J\cdot, J \cdot) = g(\cdot, \cdot)$.
Note that $J$ lifts, via the tensor product $J\otimes J$, to a map on $\Omega^2_{v,+}(M^5)$.
At each point this map has eigenvalues $\{ +1 , -1 , -1 \}$.
To see this, let $p \in M^5$ and $(x^\mu, y)_{\mu=0,\ldots,3}$ be coordinates in a neighbourhood of $p$, such that $v = \del_y$ and chosen in such a way that $dx^\mu$ is a canonical basis of $J$, i.e. $J dx^0 = dx^1$ and $J dx^2 = dx^3$.
The two-forms $e_i = dx^0 \wedge dx^i + \frac{1}{2}\epsilon_{ijk} dx^j\wedge dx^k$, $i=1,2,3$, are a local basis of $\Omega^2_{v,+}$.
One easily sees that $J e_1 = + e_1$ and $Je_{2/3} = - e_{2/3}$.

The decoupled Haydys-Witten equations arise if one sets to zero the $-1$-eigenparts of the expressions in the Haydys-Witten equations~\eqref{eq:vanishings-setting-Haydys-Witten-equations} that contain $B$, independently from the remaining terms:
\begin{definition}[Decoupled Haydys-Witten Equations]\label{def:vanishings-dHW}
\begin{align}\label{eq:vanishings-dHW}
	\tfrac{1+J}{2} \left( \sigma(B,B) + \nabla^A_v B \right) &= F_A^+ &
	\tfrac{1-J}{2} \left( \sigma(B,B) + \nabla^A_v B \right) &= 0 \\
	\delta_A^+ \tfrac{1+J}{2} B &= \imath_v F_A &
	\delta_A^+ \tfrac{1-J}{2} B &= 0
\end{align}
We denote the associated differential operator by ${\dHW[v, J]}$.
\end{definition}
In these equations the part of the $B$-terms that is located in the negative eigenspaces is decoupled from the gauge curvature, hence the name.
A direct calculation in coordinates shows that the decoupled Haydys-Witten equations are locally equivalent to~\eqref{eq:vanishings-dHW-4D} in the $4\mathcal{D}$-formalism.

To find a relation between $\HW[v]$ and $\dHW[v,J]$, split the terms in the Haydys-Witten equations that involve $B$ into their positive and negative eigenparts with respect to $J$.
To simplify notation, write $J^\pm := \tfrac{1\pm J}{2}$ as shorthand for the projection to the $\pm 1$-eigenspaces of $J$ in $\Omega^2_{v,+}$.

Denote by $\hodge$ the Hodge star operator and equip $\Omega^k(M^5, \ad E)$ with the density-valued inner product ${\langle a,b \rangle = \Tr a \wedge\hodge b}$.
Upon integration this provides the usual $L^2$-product $\langle a, b\rangle_{L^2(W)} = \int_{W^n} \langle a , b \rangle$ on $\Omega^k(W^n, \ad E)$.
The (density-valued) $L^2$-norm of the Haydys-Witten operator can be rewritten as follows
\begin{align}
	\norm{\HW[v](A,B)}^2
	&= \norm{ F_A^+ - \left( J^+ + J^- \right) \left( \sigma(B,B) + \nabla^A_v B \right)}^2
		+ \norm{ \imath_v F_A - \delta_A^+ \left( J^+ + J^- \right) B}^2  \\
	&= \norm{\dHW[v,J](A,B)}^2
		- 2 \left\langle F_A^+ , J^-\left(\sigma(B,B) + \nabla^A_v B\right) \right\rangle
		- 2 \left\langle \imath_v F_A - \delta_A^+ J^+ B, \delta_A^+ J^- B \right\rangle
\end{align}
Here we used that the $\pm 1$ eigenspaces of $J$ in $\Omega^2_{v,+}$ are orthogonal with respect to $\langle \cdot, \cdot \rangle$ to remove one of the mixed terms on the right hand side.
By an explicit calculation in a basis one further finds that the extra terms on the right hand side are in fact total derivatives:
\begin{align}
	d \Tr \left( F_A \wedge J^- B \right)
	&= \left\langle F_A^+ , \nabla^A_v J^- B \right\rangle + \left\langle \imath_v F_A, \delta_A^+ J^- B \right\rangle\\[0.5em]
	d \Tr \left( \delta_A^+ J^+ B \wedge J^- B \wedge \eta \right)
	&= \left\langle F_A^+, J^- \sigma(B,B)  \right\rangle - \left\langle \delta_A^+ J^+ B, \delta_A^+ J^- B \right\rangle
\end{align}
For the first equation one again uses orthogonality of the different subspaces of $\Omega^2$ with respect to $\langle \cdot, \cdot \rangle$ and the Bianchi identity $d_A F_A= 0$.
The second equation follows similarly after some reordering of derivatives.

With these identifications we arrive at the following.
\begin{lemma}\label{lem:vanishings-weitzenböck}
There is a Weitzenböck formula adapted to the decoupled Haydys-Witten equations
\begin{align} \label{eq:vanishings-weitzenböck}
	\int_{M^5} \norm{\HW[v](A,B)}^2  &= \int_{M^5} \left( \norm{\dHW[v,J](A,B)}^2 + d\chi \right)
\end{align}
where the exact term $\chi = \chi_1 + \chi_2$ is given by
\begin{align}
	\chi_1 &= - 2 \Tr \left( F_A \wedge J^- B \right)  \label{eq:vanishings-boundary-term-1} \\
	\chi_2 &= - 2 \Tr \left( \delta_A^+ J^+ B \wedge J^- B \wedge \eta \right) \label{eq:vanishings-boundary-term-2}
\end{align}
\end{lemma}


\section{Poly-Cylindrical Ends and Boundary Conditions}
\label{sec:vanishings-setting}

In Witten's approach to Khovanov homology, one considers manifolds of the form $\R_s \times X^3 \times \R^+_y$, which possess both a boundary and non-compact ends.
Moreover, the non-vanishing vector field $v=\del_y$ differentiates between the non-compact directions $\R_s$ and $\R^+_y$, by way of the glancing angles $\theta = \pi/2$ and $\beta = 0$, respectively.
This structure is well-described by manifolds with poly-cylindrical ends and boundaries.

In the literature, poly-cylindrical ends are typically assumed to be located at geodesic infinity.
In this case poly-cylindrical manifolds are identified with the interior of a manifold with corners that is equipped with a metric, such that any boundary point lies at infinity and the metric is of product type within some small tubular neighborhood.
Sometimes the definition is further relaxed by assuming that the metric on the manifold with corners is an \emph{exact $b$-metric} in the sense of Melrose \cite{Melrose1995, Melrose1996}.
In this latter case the metric on $M$ approaches a product metric exponentially fast, or put differently, its asymptotic expansion in $\{x^k\}_{k \in \mathbb{Z}}$ vanishes to all orders.
We expect that our results can be generalized to this situation, though we don't further pursue this here.

In our situation, we also want to include boundaries and corners at finite distance.
Therefore, when we talk about a (poly-)cylindrical end of a manifold we always take this to encompass boundaries (and corners) at both, finite and infinite geodesic distance.
We will now set up notation that will be used in the rest of this article and provide definitions for the underlying geometry.
After that, we also provide a first description of the relevant boundary conditions.

Let $M$ be a manifold with corners.
For any $p\in M$ choose a chart $(U, \phi)$ with $\phi(x) = p$ and define the \emph{depth} of $p$ by the number of components of $x$ that are zero.
This coincides with the number of boundary faces of $M$ that contain $p$ and is independent of the choice of chart.
Define the depth-$k$ stratum of $M$ to be the points
\begin{align}
	S^k(M) := \set{ p \in M \suchthat \operatorname{depth}(p) = k }
\end{align}
The interior of $M$ corresponds to $S^0(M)$, while all boundary faces of codimension one are contained in the boundary stratum $\del M := S^1(M)$, and more generally points that lie on corners of codimension $k$ are collected in $S^k(M)$.
Clearly $M = \sqcup_{k=0}^n S^k(M)$, providing a stratification of $M$.
Each stratum $S^k(M)$ is a manifold of dimension $n-k$, since corners of higher codimension are explicitly excluded.

The boundary stratum is a disjoint union of connected components $\del M = \sqcup_{i\in I\/} \del_i M$, where for simplicity we assume that $I$ is finite.
Let us emphasize that each boundary face $\del_i M$ is an open manifold without boundary.
We typically denote a boundary defining function of the boundary face $\del_i M$ by $s_i$.

\begin{definition}
A manifold with poly-cylindrical ends is a complete Riemannian manifold $(M,g)$ that is diffeomorphic to a submanifold (with corners) of a compact manifold with corners $M_0$, equipped with a metric $g_0$ that pulls back to $g$, and such that the following conditions are satisfied:
For each boundary component $\del_i M_0$ there exists a tubular neighbourhood $[0,\epsilon)_{s_i} \times \del_i M_0$, on which $g_0$ is either
\begin{align}
	g_0 = ds_i^2 + h_{\del_i M}
	\qquad
	\text{or}
	\qquad
	g_0 = \frac{ds_i^2}{s_i^2} + h_{\del_i M}
\end{align}
where $h_{\del_i M}$ is a complete metric on $\del_i M$.
Moreover, this is compatible at corners, i.e. there exists a neighbourhood $[0,\epsilon)^m \times X^{n-m}$ of each connected component of $S^k(M)$, where
\begin{align}
	g_0 = \sum_{k \in K} \frac{d s_k^2}{s_k^2} + \sum_{i \in I\setminus K} ds_i^2 + h_{X^{n-m}} 
\end{align}
\end{definition}
\begin{example}
Let $M^5 = \R_s \times X^3 \times \R^+_y$, together with the product metric $g = ds^2 + g_{X^3} + dy^2$.
Consider the compact manifold with corners $M_0 = [-1,1]_{s_0} \times X^3 \times [0,1]_{y_0}$, equipped with $g_0 = \big(\frac{ds_0}{1-\abs{s_0}}\big)^2 + h_{X^3} + \big(\frac{dy_0}{1-y_0}\big)^2$.
Then the map that takes $s_0 \mapsto s = -\operatorname{sgn}(s_0) \log(1 - \abs{s_0})$ and $y_0 \mapsto y= -\log(1- y_0)$ defines an isometry between the submanifold $U = (-1,1)_{s_0} \times X^3 \times [0,1)_{y_0}$ and $(M^5, g)$.
\end{example}

A poly-cylindrical manifold admits a convenient compact exhaustion by a family of manifolds with corner $M_\epsilon$ that mirrors the poly-cylindrical structure of $M$.
To make this precise, note that $h := \left( \prod_{k\in K} s_k^2 \right) \cdot g_0 : TM_0 \times TM_0 \to \R$ defines a Riemannian metric on $M_0$.
Let $d_h(x,y)$ denote the induced distance function and define $M_\epsilon = \set{ x \in M_0 \suchthat d_h(x, \del_i M_0) \geq \epsilon }$.
We equip $M_\epsilon$ with the restriction of $g$, making it into a compact Riemannian manifold with poly-cylindrical ends that approaches $M$ for $\epsilon \to 0$.
We will use the manifolds $M_\epsilon$ to regularize the integrals in the Weitzenböck formula \autoref{lem:vanishings-weitzenböck}.
In particular, since $M_\epsilon$ is compact, the following version of Stokes' theorem holds.
\begin{theorem}[Stokes' Theorem on Manifolds with Corners \cite{Whitney1957}]
Let $M_\epsilon$ be a compact manifold with corners, then
\begin{align}
	\int_{M_\epsilon} d \chi = \sum_{i \in I} \int_{\del_i M_\epsilon} \chi \ ,
\end{align}
where the sum on the right hand side is over all boundary faces of codimension one.
\end{theorem}
Each of the boundaries of $M_\epsilon$ can be thought of as an "$\epsilon$-displacement" of the corresponding end of $M$ into the interior.
By this we mean that for each end $[0,1)_s\times W$ of $M$, the corresponding end of $M_\epsilon$ can be identified with an embedding of $[\epsilon, 1]_s\times W \hookrightarrow [0,1)_s\times W$.
At non-compact ends, the distance between points with $s=\epsilon$ and $s=1$ is finite and proportional to $\epsilon^{-1}$, while points at $s=0$ reside at infinity.

From now on let $(M^5,g)$ denote a poly-cylindrical five-manifold and write $(M_0, g_0)$ for a suitable ambient manifold.
Throughout, we take $\del M^5$ to include both, its boundary components and non-compact cylindrical ends, i.e. we identify the boundary faces $\del_i M^5$ with the corresponding boundary components of $M_0$.

Let $[0,\epsilon)_{s} \times W^4$ be a tubular neighbourhood of a cylindrical end and denote the inward-pointing unit-normal vector field by $u$.
In case of a boundary this means that $u=\del_s$, while for a non-compact end we have $u = s \del_s$.
We say $v$ approaches a cylindrical end at constant angle if there is a tubular neighbourhood\footnotemark~on which the incidence angle $g(u ,v) = \cos \theta$ is constant.
This angle determines the natural boundary conditions at a given cylindrical end.
\footnotetext{This is equivalent to an asymptotic equivalence of $g(u,v) - \cos \theta \sim 0$ as $s_i\to 0$. 
It seems reasonable that this condition can be slightly weakened to $g(u,v) \sim \cos \theta$.}

We will now give a first overview of the relevant boundary conditions, first for boundaries and then for non-compact ends.
The classes of boundary conditions serve as organizational structure for the rest of this article.
A detailed discussion, in particular for the Nahm pole and knot singularity model solutions, will be given later, see \autoref{sec:vanishings-polyhomogeneous-expansion-of-nahm-pole-solutions} and \autoref{sec:vanishings-knot-singularity-boundary}, respectively.

At a boundary with incidence angle $g(u,v) = \cos\beta$, $\beta\neq \pi/2$, we impose $\beta$-twisted Nahm pole boundary conditions, which fix the maximal rate of growth of $(A,B)$ as one approaches the boundary.
The condition specifies that $(A,B)$ are asymptotic to a model configuration $(A^{\rho,\beta}, B^{\rho,\beta})$ of order $\mathcal{O}(y^{-1})$.
The model configuration is associated to a choice of Lie algebra homomorphism $\rho: \mathfrak{su}(2) \to \mathfrak{g}$.

It is also possible to include knot singularities of weight $\lambda\in \Gamma_{\mathrm{char}}^\vee$ in the co-character lattice of $\mathfrak{g}$.
In that case one considers the geometric blowup of a surface $\Sigma_K \subset \del M$, which introduces an additional boundary $\del_{K} M \simeq \Sigma_K \times H^2$, where $H^2$ denotes the two-dimensional hemisphere.
While the incidence angle of $v$ at this new boundary can no longer be constant, we assume that $v$ maintains a constant glancing angle $\theta$ with $\Sigma_K$.
At the knot boundary $\del_K M$, the fields must then be asymptotic to another model configuration $(A^{\lambda,\theta}, B^{\lambda,\theta})$ of order $\mathcal{O}(R^{-1})$, where $R$ denotes the radial distance to $\Sigma_K$.

The Haydys-Witten pair $(A,B)$ satisfies $\beta$-twisted Nahm pole boundary conditions at $\del M$, with knot singularity of weight $\lambda\in \Gamma_{\mathrm{char}}^\vee$ along $\Sigma_K$, if for some $\epsilon>0$
\begin{enumerate}
	\item near $\del_0 M$: $(A,B) = \big( A^{\rho,\beta} ,\; A^{\rho,\beta} \big) + \mathcal{O}\big(y^{-1+\epsilon}\big)$
	\item near $\del_K M$: $(A,B) = (A^{\lambda,\theta} , B^{\lambda,\theta} ) + \mathcal{O}(R^{-1+\epsilon})$
\end{enumerate}

Consider now the case of a non-compact end $[0,\epsilon)_{s} \times W^4$ with incidence angle $g(u,v) = \cos\theta$.
Natural boundary conditions at non-compact ends are given by stationary ($s$-invariant) solutions of the underlying differential equations.
It turns out that the specialization to $s$-invariant Haydys-Witten equations is given either by the Vafa-Witten equations or the $\theta$-Kapustin-Witten, depending on the value of $\theta$:
\begin{align}
	\HW[v](A,B) \stackrel{s\text{-inv.}}{\rightsquigarrow} \left\{ 
	\begin{array}{ll} 
		\VW(\tilde A, \tilde B, C) & \theta\equiv0 \pmod{\pi} \, , \\ 
		\KW[\theta](\tilde A, \phi) & \text{ else} 
	\end{array} \right.
\end{align}
If $\theta \neq 0 \pmod{\pi}$, there must be a non-vanishing vector field $w$ on $W^4$, such that $v=\cos\theta u + \sin\theta w$.
Moreover, we can define an orthogonal vector field $v^\perp = - \sin\theta u + \cos\theta w$.
With this notation, the four-dimensional fields above are given in terms of $(A,B)$ as follows
\begin{align}
	\tilde A &:= i^\ast A \ , \qquad
	\begin{array}{ll}
	\tilde{B} = i^\ast B \ , \ C = A_s \\[0.4em]
	\phi := i^\ast \left( \imath_{v^\perp} B + A_s \wedge w^\flat \right) 
	\end{array}
\end{align}
We will consider stationary solutions that either satisfy a finite energy condition, or that exhibit a $\beta=(\pi/2 - \theta)$-twisted Nahm pole at one of the adjacent corners of $M^5$.

We arrange the ends of $M^5$, in accordance with these boundary conditions, into the following four classes:
\begin{enumerate}[label={(\arabic*)}, noitemsep]
	\item Nahm pole boundaries $\del_{\mathrm{NP}} M$, 
	\item Knot boundaries $\del_K M$, 
	\item Kapustin-Witten ends with finite energy $\del_{\text{KW}} M$, and 
	\item Kapustin-Witten ends with Nahm poles, denoted by $\del_{\text{NP-KW}} M$.
\end{enumerate}

Let us shortly exemplify this in connection with the gauge theoretic approach to Khovanov homology.
In this situation one considers the five-manifold $M^5  = \mathbb{R}_s \times X^3 \times \mathbb{R}^+_y$ together with the vector field $v= \del_y$.
In presence of a knot $K \subset X^3$, one then moves to the geometric blowup $[M^5; \Sigma_K]$, where $\Sigma_K = \R_s \times K$.
The blown-up manifold has two types of boundary components.
First, a Nahm pole boundary $\del_{\text{NP}} M$, corresponding to those original boundary points at $y = 0$ that are not part of the $\Sigma_K$.
Second, a knot boundary $\del_K M^5$, given by the points of the unit normal bundle over the surface $\Sigma_K$.
$M^5$ also has several cylindrical ends.
A Kapustin-Witten end appears at $y \to \infty$, where the angle between $v$ and the inward pointing unit normal vector $-\del_y$ is $\theta = 0$, such that the fields approach a solution of the Vafa-Witten equations (hence, this should really be called a Vafa-Witten end).
Kapustin-Witten ends with additional Nahm poles at a corner of $M^5$ appear at $s\to \pm \infty$, where fields approach a Kapustin-Witten solutions with a Nahm pole located at the corner $\{y=0, s=\pm\infty\}$.

\begin{remark}
In case (3) the fields $(A,B)$ approach a stationary, finite energy solution of the Kapustin-Witten equations, such that one can determine the behaviour of $\chi$ with help of the vanishing theorem for $\theta$-Kapustin-Witten solutions.
In the situation of (4) we similarly ask for a stationary solution, but the fields exhibit a Nahm pole at an adjacent corner.
In this situation the fields diverge and, in particular, do not have finite energy.
\end{remark}


\section{Polyhomogeneous Expansion of Twisted Nahm Pole Solutions}
\label{sec:vanishings-polyhomogeneous-expansion-of-nahm-pole-solutions}

In this section we investigate the asymptotic behaviour of a Haydys-Witten solution $(A,B)$ in the vicinity of a Nahm pole boundary $\del_{\mathrm{NP}} M$, i.e. under the assumption that $(A,B)$ satisfy twisted Nahm pole boundary conditions without knot singularities.
The goal of this analysis is to extract the leading order behaviour of $F_A$ and $\delta_A^+ J^+ B$ near a Nahm pole boundary, which is later used to determine the asymptotics of $\chi$.

The analysis goes back to an essentially identical discussion for the $\theta = \pi/2$ version of the Kapustin-Witten equations with untwisted Nahm pole boundary conditions by Mazzeo and Witten in \cite{Mazzeo2013a}.
There, a Weitzenböck formula on $\R^3\times \R^+$ very similar to \autoref{lem:vanishings-weitzenböck} -- in fact a more restrictive one -- was established and then utilized to determine the analytic properties of the Nahm pole boundary conditions in the first place.
While Mazzeo and Witten already touched on the leading and constant orders for general four-manifolds, Siqi He has later expanded on the constraints that arise for subleading orders on general four-manifolds and unveiled a deep relation to the geometry of the boundary \cite{He2018a}.

To set the stage, let $U=W^4 \times [0,1)_y$ be a tubular neighbourhood of a Nahm pole boundary $\del_{\mathrm{NP}} M$.
Write $u=\del_y$ for the inward-pointing, unit normal vector field.
As always, assume that the incidence angle determined by $g(u,v) = \cos\beta$ is constant on all of $U$.
Whenever $\beta \neq 0$ there is a non-vanishing unit vector field $w$ on $W^4$, such that $v= \sin\beta w + \cos\beta u$.

Let us shortly recall some additional details about the Nahm pole boundary conditions on $\del_{\mathrm{NP}} M$, see \cite{Mazzeo2013a, Mazzeo2017}.
Let $\phi_\rho \in \Omega^2_{v,+}(U, \ad E)$ be some fixed boundary configuration associated to the Nahm pole boundary conditions on $U$.
This means that it satisfies $\phi_\rho - \sigma(\phi_\rho, \phi_\rho) = 0$, such that its components $(\phi_\rho)_i = \mathfrak{t}_i$, $i=1,2,3$, span an $\mathfrak{su}(2)$ subalgebra inside of $\mathfrak{g}$.
There is a unique corresponding two-form that satisfies $\phi_\rho^\tau + \sigma(\phi_\rho^\tau , \phi_\rho^\tau) = 0$.
Here $\tau=(132)$ denotes an anticyclic permutation that acts on components by $(\phi_\rho^\tau)_i = (\phi_\rho)_{\tau(i)}$ and is related to a reversal of orientation on $\Omega^2_{v,+}(U)$.
If $\beta \neq 0$, there is an isomorphism by which we can view elements of $\Omega^2_{v,+}(U)$ as the one-forms dual to a rank 3 subbundle $\Delta_{(u,v)}^\perp \subset TU$.
Here, $\Delta_{(u,v)}^\perp$ is the orthogonal complement of the subbundle spanned by $u$ and $v$.
With that in mind, the Nahm pole boundary condition states that as $y\to 0$, $A$ and $B$ are asymptotic to 
\begin{align}
	A^{\rho,\beta} = \sin\beta \; \frac{\phi_\rho^\tau}{y} \ ,  \qquad B^{\rho,\beta} = \cos\beta\; \frac{\phi_\rho}{y} \ .
\end{align}
Near a Nahm pole, the Haydys-Witten equations are amended by the following gauge fixing condition
\begin{align}
	d_{A^0}^{\hodge_5} (A - A^0) + \sigma(B^{\mathrm{NP}}, B - B^{\mathrm{NP}}) &= 0 \ ,
\end{align}
making the equations into an elliptic system of differential equations.

Ultimately, we want to investigate the leading order of $\chi_1= -2 \Tr( F_A \wedge J^- B)$ and $\chi_2 = -2 \Tr( \delta_A^+ J^+ B \wedge J^- B \wedge \eta)$ near $\del_{\mathrm{NP}} M$.
Since this involves derivatives and products of $A$ and $B$, we need to study the subleading orders in a polyhomogeneous expansion of $A$ and $B$ as $y\to 0$.

Write $A= y^{-1} \sin\beta\ \phi_\rho^\tau + \omega + a$ and $B = y^{-1} \cos\beta\ \phi_\rho + b$, where by assumption of the Nahm pole boundary condition $a,b \in \mathcal{O}(y^{-1+\epsilon})$, while $\omega$ is a gauge connection on the restriction of $E$ to $W^4$, smoothly extended into the bulk.
If $(A,B)$ is a solution of the Haydys-Witten equations that satisfies Nahm pole boundary conditions, then an elliptic regularity theorem Mazzeo and Witten \cite{Mazzeo2013a} states that there exist polyhomogeneous expansions
\begin{align}
	a &\sim \sum_{(\alpha, k) \in \Delta_0} a_{\alpha,k}(x^\mu)\; y^{\alpha} (\log y)^k \ , &
	b &\sim \sum_{(\alpha, k) \in \Delta_0 } b_{\alpha, k}(x^\mu)\; y^\alpha (\log y)^k &
	(y&\to 0)
	\label{eq:vanishings-general-polyhomogeneous-expansion}
\end{align}
where the indicial set $\Delta_0 \subset \C \times \N$ is bounded from the left by $1$ (in particular $a_{0,k} = b_{0,k} = 0$).

Plugging $A= y^{-1} \sin\beta\ \phi_\rho^\tau + \omega + a$ and $B = y^{-1} \cos\beta\ \phi_\rho + b$ into the Haydys-Witten equations then determines a sequence of constraints on the functions $a_{\alpha,k}$ and $b_{\alpha,k}$.
For our purposes it will suffice to determine $\chi$ up to $\mathcal{O}(y^\delta)$, $\delta >0$, so we only need to consider an expansion up to order $y^{2+\delta}$.
We will see below that for generic values of $\beta$ only the following terms in the expansion are non-zero.
\begin{align}
	a &= a_{1,1}\ y \log y + a_{1,0}\ y + a_{2,1}\ y^2 \log y + a_{2,0}\ y^2 + \mathcal{O}(y^{5/2}) \\
	b &= b_{1,1}\ y \log y + b_{1,0}\ y + b_{2,1}\ y^2 \log y + b_{2,0}\ y^2 + \mathcal{O}(y^{5/2})
\end{align}
Moreover, the coefficient functions satisfy additional constraints that link them to the Riemannian curvature of the manifold.
In particular, the $\log y$-terms vanish if sectional curvature is flat in any plane that contains $w$.

To make this more precise, we need to establish analogues of the tools that were used in the discussion of indicial roots in \cite{Mazzeo2013a} and \cite{He2018a}.
First, the fibers of $\ad E$ decompose under the action of $(\phi_\rho)_i = \mathfrak{t}_i \in \mathfrak{su}(2)_\mathfrak{t}$ as a direct sum of spin $j$ representations.
We denote the associated vector bundles by $V_j$, $j \in J$.
Since we consider regular Nahm pole conditions, only positive integers $j$ appear; in fact for $G=SU(N)$ one finds $J = \{1, \ldots, N-1\}$ \cite{Mazzeo2013a}.
Second, there is an action by rotations $\mathfrak{so}(3) \simeq \mathfrak{su}(2)_{\mathfrak{s}}$ on $\Omega^2_{v,+}(W^4)$ and $(\Delta_{(u,v)}^\perp)^\ast$.
For fixed $j$, the bundles $\Omega^2_{v,+}(W^4, \ad E)$ and $\Hom(\Delta_{(u,v)}^\perp, V_j) \subset \Omega^1(W^4, \ad E)$ thus further decomposes under a combined $\mathfrak{su}(2)_\mathfrak{f}$ action, generated by $\mathfrak{f}_i = \mathfrak{t}_i + \mathfrak{s}_i$, into a direct sum of bundles $V_{j}^-$, $V_j^0, $ and $V_j^+$ of total spin $j-1$, $j$ and $j+1$, respectively.
To simplify notation, we will furthermore write $V_{j}^{w} := C^\infty(W^4, V_j)\ w^\flat$, where we let $\mathfrak{su}(2)_\mathfrak{s}$ acts trivially.
All in all, the relevant $\ad E$-valued differential forms decompose as
\begin{align}
	\begin{split}
	\Omega^1(W^4, \ad E) &=  \bigoplus_{j\in J} \bigg(\ V_{j}^{w} \oplus V_{j}^{-} \oplus V_{j}^{0} \oplus V_{j}^{+} \ \bigg) \\
	\Omega^2_{v,+}(W^4, \ad E) &= \bigoplus_{j\in J} \bigg( V_{j}^{-} \oplus V_{j}^{0} \oplus V_{j}^{+} \bigg)
	\end{split}
	\label{eq:vanishings-su2-decomposition}
\end{align}
This decomposition is helpful, because the leading order terms $A^{\mathrm{NP}} = y^{-1} \sin\beta \phi_\rho^\tau$ and $B^{\mathrm{NP}} = y^{-1} \cos\beta \phi_\rho$ are invariant under the action of $\mathfrak{su}(2)_\mathfrak{f}$ and, as a consequence, the constraints for $a_{\alpha,k}$ and $b_{\alpha,k}$ decompose into their $V_j^\eta$-valued components.
However, there is an additional subtlety regarding invariance of the Nahm pole terms under the action of $\mathfrak{su}(2)_\mathfrak{f}$.
To see this, let rotations $\mathfrak{su}(2)_\mathfrak{s}$ act on $\Omega^2_{v,+}$ with the "standard" orientation on $(e_1, e_2, e_3)$ and assume this basis is identified with $(dx^1, dx^2, dx^3) \in (\Delta_{(u,v)}^\perp)^\ast$.
Then $\phi_\rho =\sum_i \mathfrak{t}_i e_i$ or equivalently $\sum_i \mathfrak{t}_i dx^i$ is an element of the trivial representation $V_1^- \subset \Omega^2_{v,+}(W^4,\ad E) \simeq \Hom(\Delta_{(u,v)}^\perp, \ad E)$, such that $B^{\mathrm{NP}}$ is invariant.
Unfortunately, under this identification $\phi_\rho^\tau = \sum_i \mathfrak{t}_{\tau(i)} dx^i = \mathfrak{t}_1 dx^1 + \mathfrak{t}_2 dx^3 + \mathfrak{t}_3 dx^2$ is not invariant, since orientation reversal $\phi_\rho \mapsto \phi_\rho^\tau$ does not preserve $V_1^-$.
To remedy this, we need rotations to act on the $\Delta_{(u,v)}^\perp$-part of $A$ with the opposite orientation $(dx^1, dx^3, dx^2)$.
With that choice $\phi_\rho^\tau$ is an element of the trivial representation $\tilde{V}_1^-$, with respect to the analogous (but different) decomposition
\begin{align}
	\Omega^1(W^4, \ad E) = \bigoplus_{j \in J} \left( \tilde{V}_j^w \oplus \tilde{V}_j^- \oplus \tilde{V}_j^0 \oplus \tilde{V}_j^+ \right) \ .
\end{align} 
$A^{\mathrm{NP}}$ becomes invariant under $\mathfrak{su}(2)_\mathfrak{f}$ only with that choice.
We denote the restriction of a given differential form $\alpha$ to one of the subspaces $V_j^\eta$ or $\tilde{V}_j^\eta$ in this decomposition by $(\alpha)^{j,\eta}$ or $(\alpha)^{\widetilde{j,\eta}}$, respectively.

\begin{example}
In a local basis $(e_i)_{i=1,2,3}$ of $\Omega^2_{v,+}(U, \ad E)$ the subbundles $V_1^\eta \subset \Omega^2_{v,+}(W^4, \ad E)$ are given by:
\begin{align}
	V_1^- &= \spans{ \mathfrak{t}_1 e_1 + \mathfrak{t}_2 e_2 + \mathfrak{t}_3 e_3 } \ , \\
	V_1^0 &= \spans{ \mathfrak{t}_i e_j - \mathfrak{t}_j e_i}_{i\neq j} \ , \\
	V_1^+ &= \spans{ \mathfrak{t}_1 e_1 - \mathfrak{t}_2 e_2,\ \mathfrak{t}_1 e_1 - \mathfrak{t}_3 e_3,\ \mathfrak{t}_i e_j + \mathfrak{t}_j e_i}_{i\neq j} \ .
\end{align}
Substituting $e_i$ by $dx^i$ provides the corresponding basis under the isomorphism of $\Omega^2_{v,+}(W^4, \ad E)$ and $\Hom(\Delta_{(u,v)}^\perp, \ad E) \subset \Omega^1(W^4,\ad E)$.
This is the correct decomposition for the action of $\mathfrak{su}(2)_\mathfrak{f}$ on $B$ and its expansions $\phi_\rho$, $b_{\alpha,k}$.
In contrast, for the decomposition of $\Omega^1(W^4,\ad E)$ with respect to the reverse orientation, a corresponding basis of $\tilde{V}_1^\eta$, $\eta \in \{-,0,+\}$, is obtained from the one above by replacing $e_i$ with $dx^{\tau(i)}$.
For example $\tilde{V}_1^- = \spans{\mathfrak{t}_1 dx^1 + \mathfrak{t}_3 dx^2 + \mathfrak{t}_2 dx^3}$.
This is the correct basis for the action of $\mathfrak{su}(2)_\mathfrak{f}$ on $A$ and its expansions $\phi_\rho^\tau$, $a_{\alpha,k}$.
\end{example}

The upcoming lemma and its proof specify the algebraic constraints the Haydys-Witten equations put on the coefficient functions in the polyhomogeneous expansion up to order $y^2$.
See \cite{He2018a} for a similar and more extensive investigation of these constraints in the case of $\pi/2$-Kapustin-Witten solutions with Nahm pole boundary conditions.

As becomes clear during the proof, we have to exclude a finite set of angles at which certain spin-$j$ modes of $\imath_w a$, i.e. the components $(a_{\alpha,k})^{j,w}$, become free parameters.
Explicitly, these values are given by:
\begin{align}
	\cos 2\beta \in \left\{-\tfrac{2j+3}{(j+1)^2},\ \tfrac{2j-1}{j^2},\ -\tfrac{3j^2-4j+3}{j(j+1)^2},\ -\tfrac{3j^2+2j-4}{j^2(j+1)} \right\}_{j\in J}
\end{align}
The statement of the following lemma holds whenever $\beta$ is not of that form.
\begin{lemma} \label{lem:vanishings-expansion-of-nahm-pole-solutions}
Assume $A= y^{-1} \sin\beta\ \phi_\rho^\tau + \omega + a$ and $B = y^{-1} \cos\beta\ \phi_\rho + b$ are a solution of the Haydys-Witten equations on $W^4 \times \R^+_y$ with respect to the non-vanishing vector field $v=\sin\beta w + \cos\beta \del_y$.
Then $\omega$ pulls back under $\phi_\rho$ to the Levi-Civita connection on $TW^4$ and $\nabla^\omega_w \phi_\rho = 0$.
Furthermore, if $\beta$ is generic, then $a_{1,k} =b_{1, k} = a_{2, k} =  b_{2, k} = 0$ for all $k\geq 2$.
The solution is smooth up to the boundary (i.e. $a_{1,1} = b_{1,1}=0$), if and only if the Riemannian curvature satisfies $(\cos\beta \imath_w F_\omega + \imath_w \hodge_4 F_\omega)\restrict{V_1^+/\tilde{V}_1^0} = 0$.
Finally, if $F_\omega = 0$, the polyhomogeneous expansion up to order $y^{2+\delta}$, $\delta>0$, reduces to
\begin{align}
	a &= y \sin\beta\ (C^{1,+})^\tau + y^2 \left( \sin\beta\ C^{2,+} + \cos\beta\ D^{1,-} \right)^\tau + \mathcal{O}(y^{2+\delta}) \\
	b &= y \cos\beta\ C^{1,+} + y^2 \left( \cos\beta\ C^{2,+} - \sin\beta\ D^{1,-} \right) + \mathcal{O}(y^{2+\delta})
\end{align}
where $C^{1,+} \in V_1^+ / \tilde{V}_1^0$, $D^{1,-}\in V_1^-$, and $C^{2,+} \in V_2^+ / \tilde{V}_2^0$ remain unconstrained.
\end{lemma}
\begin{corollary} \label{cor:vanishings-expansion-of-field-strength}
Denote by $i_y : W^4 \hookrightarrow W^4 \times \R^+_y$ inclusion at $y$.
For generic $\beta$ and if $F_\omega = 0$, the total field strength pulls back to
\begin{align}
	i_y^\ast ( F_A )= {}
		& \phantom{{}+{}} y^{-2} \left( \sin^2\beta\ [ \phi_\rho^\tau \wedge \phi_\rho^\tau ] \right)
		+ y^0  \left( \sin^2\beta\ [\phi_\rho^\tau \wedge (C^{1, +})^\tau] \right) \\
		& + y^1	 \left( \sin\beta\ d_\omega (C^{1,+})^\tau + \sin^2\beta [\phi_\rho^\tau \wedge (C^{2,+})^\tau] + \sin\beta\cos\beta\ [\phi_\rho^\tau \wedge (D^{1,-})^\tau]  \right)
		+ \mathcal{O}(y^{1+\delta})
\end{align}
Meanwhile, the $dx^2$- and $dx^3$- components of $\delta_A^+ J^+ B$ are specified by
\begin{align}
	i_y^\ast ( \delta_A^+ J^+ B ) \propto {}
		& \phantom{{}+{}} y^{-2} \sin\beta\cos\beta \left( [(\phi_\rho^\tau)_3, (\phi_\rho)_1] dx^2 - [(\phi_\rho^\tau)_2, (\phi_\rho)_1] dx^3 \right) \\
		& + y^{0} \sin\beta\cos\beta\
			\begin{aligned}[t]
				\Big(
					&\left( [(\phi_\rho^\tau)_3 , (C^{1,+})_1] + [(C^{1,+})^\tau_3, (\phi_\rho)_1] \right) dx^2 \\
				 - &\left( [(\phi_\rho^\tau)_2 , (C^{1,+})_1] + [(C^{1,+})^\tau_2, (\phi_\rho)_1] \right) dx^3
				\Big)
			\end{aligned}
			\\
		& + y^{1}
			\begin{aligned}[t]
				\Big(
					&\left(\cos\beta\ \nabla^\omega_3 (C^{1,+})_1 + \sin\beta\cos\beta\ [(\phi_\rho^\tau)_3, (C^{2,+})_1] - \sin^2\beta\ [(\phi_\rho^\tau)_3, (D^{1,-})_1] \right) dx^2 \\
					- &\left(\cos\beta\ \nabla^\omega_2 (C^{1,+})_1 + \sin\beta\cos\beta [(\phi_\rho^\tau)_2, (C^{2,+})_1] - \sin^2\beta\ [(\phi_\rho^\tau)_2, (D^{1,-})_1] \right) dx^3
				\Big)
			\end{aligned}
			\\
		&+ \mathcal{O}(y^{1+\delta})
\end{align}
\end{corollary}
The rest of this section is occupied with the proof of \autoref{lem:vanishings-expansion-of-nahm-pole-solutions}.

\begin{proof}
The proof is analogous to the analysis in \cite{He2018a}, which covers the case $\beta = 0$.
Hence, assume that $\beta \neq 0$, in which case there is an isomorphism $\Omega^2_{v,+}(U, \ad E) \simeq \Hom( \Delta_{(u,v)}^\perp, \ad E)$.
We use this equivalence to first re-express the Haydys-Witten equations~\eqref{eq:vanishings-setting-Haydys-Witten-equations} purely in terms of one-forms.
This simplifies the evaluation of the formulas later on.

Fix a gauge in which  $A_y = 0$ and choose as reference connection $A^0 = y^{-1} \sin\beta\ \phi_\rho^\tau + \omega$, where $\omega$ is the pullback of a connection on $E \to W^4$.
Note that when we view $B$ as an element of $\Hom(\Delta_{(u,v)}^\perp, \ad E) \subset \Omega^1(W^4 \times \R^+_y, \ad E)$, it does not have a $dy$-component: in coordinates $(s,x^i, y)_{i=1,2,3}$ with $v=\sin\beta \del_s + \cos\beta \del_y$, the two-form ${B = \sum_{i=1}^3 \phi_i e_i}$ is identified with the one-form $\sum_{i=1}^3 \phi_i dx^i$.
Since the coefficients $\phi_\rho^\tau$, $\phi_\rho$, $a_{\alpha, k}$ and $b_{\alpha,k}$ are independent of $y$ and as they don't have $dy$-components, we can view them as elements of $\Omega^1(W^4, \ad E)$.
Motivated by this, we slightly abuse notation in the upcoming version of the Haydys-Witten equations~\eqref{eq:vanishings-Haydys-Witten-equations-rewriting-1}~-~\eqref{eq:vanishings-Haydys-Witten-equations-rewriting-4} and treat $A$ and $B$ as one-forms on $W^4\times\{y\}$.
(Specifically, we implicitly take the pullback and drop any $dy$-components in $F_A$ and $d_A B$.)
With that understood, the Haydys-Witten equations can be brought into the form
\begin{align}
	\del_{y\/} A &= - \imath_w \left( d_A B - \sin\beta\ \hodge_4 ( F_A - [B\wedge B] ) - \cos\beta\ \hodge_4 d_A B \right)
		\label{eq:vanishings-Haydys-Witten-equations-rewriting-1} \\
	\del_{y\/} B &= \phantom{\mathrel{-}} \imath_w \left( F_A + \cos\beta\ \hodge_4 ( F_A - [B\wedge B] ) - \sin\beta\ \hodge_4 d_A B \right)
		\label{eq:vanishings-Haydys-Witten-equations-rewriting-2} \\
	\del_{y\/} (\imath_w A) &= \phantom{\mathrel{-}} d_A^{\hodge_4} B
		\label{eq:vanishings-Haydys-Witten-equations-rewriting-3} \\
	0 &= \phantom{\mathrel{-}} d_\omega^{\hodge_4} a + \sin\beta \hodge_4 [ \phi_\rho^\tau \wedge \hodge_4 y^{-1} a] + \cos\beta  \hodge_4 [ \phi_\rho \wedge \hodge_4 y^{-1} b]
		\label{eq:vanishings-Haydys-Witten-equations-rewriting-4}
\end{align}

Plugging the polyhomogeneous expansions of $A$ and $B$ into the terms that appear in these equations yields:
\begin{align}
		F_A = {} &
		y^{-2} \big( \sin^2\beta\ [ \phi_\rho^\tau \wedge \phi_\rho^\tau ] + \sin\beta\ \phi_\rho^\tau \wedge dy \big) \
		+\ y^{-1} \big( \sin\beta\ d_{\omega} \phi_\rho^\tau \big) \
		+\ y^{0}  \big( F_\omega \big) \\[0.5em]
		&{} + \sum_{(\alpha, k) \geq 0} y^\alpha (\log y)^k \left( d_\omega a_{\alpha,k} + \sin\beta [\phi_\rho^\tau \wedge a_{\alpha+1, k}] - (\alpha+1) a_{\alpha+1, k} \wedge dy - (k+1) a_{\alpha+1, k+1} \wedge dy \right)
	\\[1em]
	[B \wedge B ] = {} & {}
		y^{-2} \left( \cos^2\beta\ [\phi_\rho \wedge \phi_\rho] \right)
		+ \sum_{(\alpha,k) \geq 0} y^{\alpha} (\log y)^k \left( \cos\beta [\phi_\rho \wedge b_{\alpha+1, k}] \right)
	\\[1em]
	d_A B = {} &
		y^{-2} \left( \sin\beta\cos\beta\ [\phi_\rho^\tau \wedge \phi_{\rho}] + \cos\beta\ \phi_\rho \wedge dy \right)
		+ y^{-1} \left( \cos\beta\ d_{\omega} \phi_\rho \right) \\
		& + \sum_{(\alpha,k) \geq 0} y^\alpha (\log y)^k
		\begin{multlined}[t][0.67\textwidth]
			\left( d_\omega b_{\alpha, k} + \sin\beta\ [\phi_\rho^\tau \wedge b_{\alpha+1, k}] + \cos\beta\ [\phi_\rho \wedge a_{\alpha+1, k}]
			\right. \\ \left.
			- (\alpha +1) b_{\alpha+1, k} \wedge dy - (k+1) b_{\alpha+1, k+1} \wedge dy \right)
		\end{multlined}
	\\[1em]
	d_A^{\hodge_4} B = {} & {}
		y^{-1} \left( \cos\beta\ d_{\omega}^{\hodge_4} \phi_\rho \right)
		+ \sum_{(\alpha,k) \geq 0} y^\alpha (\log y)^k
		\begin{multlined}[t][0.5\textwidth]
			\left( d_\omega^{\hodge_4} b_{\alpha, k} + \sin\beta\ \hodge_4 [\phi_\rho^\tau \wedge \hodge_4 b_{\alpha+1,k}]
			\right. \\ \left.
			- \cos\beta\ \hodge_4 [\phi_\rho \wedge \hodge_4 a_{\alpha+1, k}] \right)
		\end{multlined}
\end{align}
Here we write $(\alpha,k)\geq 0$ for pairs of exponents with $\Re \alpha \geq 0$.
Using $(\del_y A)_i = -(F_A)_{iy}$ and $(\del_y B)_i = -(d_A B)_{iy}$, we can now read off the contributions to the Haydys-Witten equations from each of these expressions, order by order in $y^{\alpha} (\log y)^k$.

\vspace{0.5em}
\underline{$\mathcal{O}(y^{-2})$}

At order $\mathcal{O}(y^{-2})$ only equations \eqref{eq:vanishings-Haydys-Witten-equations-rewriting-1} and \eqref{eq:vanishings-Haydys-Witten-equations-rewriting-2} are non-trivial.
Using that $\imath_w (\hodge_4 [\phi_\rho^\tau \wedge \phi_\rho]) = \thalf \imath_w ( \hodge_4 [\phi_\rho^\tau \wedge \phi_\rho^\tau] + \hodge_4  [\phi_\rho \wedge \phi_\rho] )$, the Haydys-Witten equations then state:
\begin{align}
	\phi_\rho^\tau &= - \imath_w \left( \hodge_4 [\phi_\rho^\tau \wedge \phi_\rho^\tau] \right) \\
	\phi_\rho &= \phantom{\mathrel{-}} \imath_w \left( \hodge_4 [\phi_\rho \wedge \phi_\rho] \right)
\end{align}
This is satisfied, since under the isomorphism $\Omega^2_{v,+}(U,\ad E) = \Hom(\Delta_{(u,v)}^\perp, \ad E)$ the expression $\imath_w (\hodge_4 [\phi_\rho \wedge \phi_\rho])$ is identified with $\sigma(\phi_\rho, \phi_\rho)$. 
By construction of the Nahm pole boundary condition $\phi_\rho - \sigma(\phi_\rho, \phi_\rho) = 0$, while the same is true for $\phi_\rho^\tau$ with the opposite sign.

\vspace{0.5em}
\underline{$\mathcal{O}(y^{-1})$}

The first additional constraint arises at $\mathcal{O}(y^{-1})$, where we find
\begin{align}
	0	&= \phantom{\mathrel{-}} \imath_w \left( \cos\beta\ d_\omega \phi_\rho - \sin^2\beta\ \hodge_4 d_\omega \phi_\rho^\tau - \cos^2\beta\ \hodge_4 d_\omega \phi_\rho \right) \\
	0	&= - \imath_w \left( \sin\beta\ d_\omega \phi_\rho^\tau + \sin\beta\cos\beta\ \hodge_4 d_\omega \phi_\rho^\tau - \sin\beta\cos\beta\ \hodge_4 d_\omega \phi_\rho \right) \\
	0 &= \cos\beta\ d_\omega^{\hodge_4} \phi_\rho
\end{align}
These equations are satisfied if the tensor\footnotemark~$\nabla^\omega \phi_\rho$ vanishes.
\footnotetext{%
Here, as always, $\nabla^\omega$ is the product connection on~${\Omega^1( W^4\times \R^+_y , \ad E)}$ that is induced by the Levi-Civita connection and the gauge connection $\omega$.
}
This means that $\phi_\rho$ intertwines the Levi-Civita and gauge connection on $TW^4$:
On the one hand, $\omega$ must pull back under $\phi_\rho : \Delta_{(u,v)}^\perp \to \ad E$ to the Levi-Civita connection on $\Delta_{(u,v)}^\perp$.
This also implies that the restriction of $\omega$ to $\Delta_{(u,v)}^\perp \subset TW^4$ is valued in $\rho( \mathfrak{su}(2) ) = V_1 \subset \ad E$.
On the other hand, the $w^\flat$-component of $\omega$ must be such that $\nabla^\omega_w \phi_\rho = 0$, i.e. $\phi_\rho$ extends along $w$ by parallel transport.
Since the Levi-Civita connection acts trivially on $\ad E$, it preserves $V_1 \subset \ad E$ and we find that also $\imath_w \omega \in V_1$.
We note that $\imath_w \hodge_4 F_\omega$ and $\imath_w F_\omega$ are then elements of $V_1^- \oplus V_1^0 \oplus V_1^+$ and are determined by the Riemannian curvature of $W^4$.
Furthermore, the exterior covariant derivative satisfies the following mapping properties (cf. \cite[p.5]{Henningson2012}):
\begin{align}
	\begin{split}
	\imath_w \hodge_4 d_\omega &:
		V_j^w \to 0 \ , \quad V_{j}^- \to V_{j}^0 \ , \quad   V_{j}^0 \to V_{j}^{-} \oplus V_{j}^0 \oplus V_{j}^+  \ , \quad  V_{j}^+ \to V_{j}^0 \oplus V_{j}^+ \\
	\imath_w d_\omega &: V_j^w \to V_j^- \oplus V_j^0 \oplus V_j^+ \ , \quad V_{j}^\eta \to V_{j}^\eta , \ \eta \in \{-,0,+\}
	\end{split}
	\label{eq:vanishings-exterior-derivative-mapping-properties}
\end{align}
With a view at the expansions of $F_A$ and $d_A B$, it becomes clear that these identities will be helpful in the upcoming analysis, starting at order $\mathcal{O}(y^0 (\log y)^k)$.

\vspace{0.5em}
\underline{$\mathcal{O}(y^0 (\log y)^k )$}

Examining the higher order expansions of the right hand side of~\eqref{eq:vanishings-Haydys-Witten-equations-rewriting-1}~-~\eqref{eq:vanishings-Haydys-Witten-equations-rewriting-4}, one finds that the equations will from now on involve terms of the form $\imath_w \hodge_4 [\phi_\rho \wedge \cdot]$ and its analogue with respect to $\phi_\rho^\tau$.
These are global versions of the spin-spin operator $\mathfrak{J}:= \mathfrak{s} \cdot \mathfrak{t}$, as they act pointwise by $\sum_{a=1}^3 \mathfrak{s}_a \otimes \mathfrak{t}_a$.
Indeed, taking into account that rotations act on $a_{\alpha,k}$ and $b_{\alpha,k}$ with opposite orientation one finds
\begin{align}
	\mathfrak{J} a_{\alpha,k} &= - \imath_w \hodge_4 [\phi_\rho^\tau \wedge a_{\alpha,k}] &
	\mathfrak{J} b_{\alpha,k} &= \imath_w \hodge_4 [\phi_\rho \wedge b_{\alpha,k}]
\end{align}
The vector bundles $V_j^\eta$, $\eta \in \{-, 0, + \}$, are eigenspaces of the spin-spin operator with eigenvalues $j+1$, $1$, and $-j$.
Meanwhile, $\mathfrak{J}$ sends $V_j^w$ to zero.
We will also use that $\imath_w [\phi_\rho \wedge \cdot]$ maps $V_j^w \to V_j^0$ and $V_{j}^\eta \to 0$, $\eta \in \{-,0,+\}$.

Equipped with this, we can now proceed to the constraints that arise for the tower of $\log y$ terms at order $y^0$.
Hence, consider equations~\eqref{eq:vanishings-Haydys-Witten-equations-rewriting-1}~-~\eqref{eq:vanishings-Haydys-Witten-equations-rewriting-4} at $\mathcal{O}\big((\log y)^k\big)$, where $k$ is the largest integer such that $(1, k) \in \Delta_0$.
The existence of such a maximal value is part of the defining properties of a polyhomogeneous expansion.
For the moment assume that $k\geq 2$, we will treat the cases $k=1$ and $k=0$ in the end.

We now introduce the following linear combinations of $V_{j}^-\oplus V_j^0 \oplus V_j^+$-parts of $a_{\alpha,k}$ and $b_{\alpha,k}$:
\begin{align}
	c_{\alpha,k} &:= \sin\beta\ a_{\alpha,k}^\tau + \cos\beta\ b_{\alpha,k}\ , \\
	d_{\alpha,k} &:= \cos\beta\ a_{\alpha,k}^\tau - \sin\beta\ b_{\alpha,k}\ .
\end{align}
Note, in particular, that $c_{\alpha,k} $ and $d_{\alpha,k}$ do not include the component $(a_{1,k})^{j,w}$.
Appropriate linear combinations of~\eqref{eq:vanishings-Haydys-Witten-equations-rewriting-3}~and~\eqref{eq:vanishings-Haydys-Witten-equations-rewriting-4} then state
\begin{align}
	\sin\beta\ \imath_w a_{1,k} &= \cos\beta\ \hodge_4 [\phi_\rho \wedge \hodge_4 c_{1,k}] + \sin\beta\ \hodge_4 [\phi_\rho \wedge \hodge_4 d_{1,k} ]  \\
	\cos\beta\ \imath_w a_{1,k} &= \sin\beta\ \hodge_4 [\phi_\rho^\tau \wedge \hodge_4 c_{1,k}] - \cos\beta\ \hodge_4 [\phi_\rho^\tau \wedge \hodge_4 d_{1,k}]
\end{align}
These equations can be brought in a slightly more helpful form.
Note that the operator $[\phi_\rho \wedge \hodge_4 [\phi_\rho \wedge \hodge_4\, \cdot \ ]]$ annihilates $V_j^\pm$ and acts on $V_j^0$ as $\sum_i [\mathfrak{t}_i , [\mathfrak{t}_i, \cdot]]$, which is just (the negative of) the quadratic Casimir on $V_j$.
Applying $[\phi_\rho \wedge \cdot]$ to the equations thus simplifies the operation on the right hand side to a simple multiplication by $-j(j+1)$.
In complete analogy $[\phi_\rho^\tau \wedge \hodge_4 [\phi_\rho^\tau \wedge \hodge_4\, \cdot \ ]]$ annihilates $\tilde{V}_j^\pm$ and acts on $\tilde{V}_j^0$ by multiplication with $-j(j+1)$.
This results in~\eqref{eq:vanishings-logy-3}~and~\eqref{eq:vanishings-logy-4} below.
Including appropriate linear combinations of \eqref{eq:vanishings-Haydys-Witten-equations-rewriting-1}~and~\eqref{eq:vanishings-Haydys-Witten-equations-rewriting-2}, the coefficient functions in the polyhomogeneous expansion have to be compatible with the following four conditions.
\begin{align}
	c_{1,k} + \mathfrak{J} c_{1,k}
		&= \phantom{\mathrel{-}} \sin\beta\cos\beta\ \left( \imath_w  [\phi_\rho^\tau \wedge a_{1,k}] - \imath_w [\phi_\rho \wedge a_{1,k}] \right) \label{eq:vanishings-logy-1} \\
	d_{1,k} - \mathfrak{J} d_{1,k}
		&= - \cos^2\beta\ \imath_w [\phi_\rho^\tau \wedge a_{1,k}] - \sin^2\beta\ \imath_w [\phi_\rho \wedge a_{1,k}] \label{eq:vanishings-logy-2} \\
	j(j+1) \left( \cos\beta\ c_{1,k} + \sin\beta\ d_{1,k} \right)^{j,0} &= \phantom{\mathrel{-}} \sin\beta\ \imath_w [\phi_\rho \wedge a_{1,k}] \label{eq:vanishings-logy-3} \\
	j(j+1) \left( \sin\beta\ c_{1,k} - \cos\beta\ d_{1,k} \right)^{\widetilde{j,0}} &= \phantom{\mathrel{-}} \cos\beta\ \imath_w [\phi_\rho^\tau \wedge a_{1,k}] \label{eq:vanishings-logy-4}
\end{align}
Observe that the terms on the left only depend on $c_{1,k}$ and $d_{1,k}$ and thus only on the $V_j^- \oplus V_j^0 \oplus V_j^+$ components of $a_{1,k}$ and $b_{1,k}$.
Meanwhile, the right hand side only depends on the remaining $V_j^w$ component $(a_{1,k})^{j,w}$.
Moreover, all terms on the right are necessarily located in $V_j^0 + \tilde{V}_j^0$ (fiberwise sum of subspaces).
To solve the equations we distinguish between the following six cases:
\begin{multicols}{3}
\begin{enumerate}[label=(\Roman*)]
	\item $V_j^- / \tilde{V}_j^0$
	\item $V_j^+ / \tilde{V}_j^0$
	\item $V_j^- \cap \tilde{V}_j^0$
	\item $V_j^+ \cap \tilde{V}_j^0$
	\item $V_j^0 / \tilde{V}_j^0$
	\item $V_j^0 \cap \tilde{V}_j^0$
\end{enumerate}
\end{multicols}
\vspace*{-0.7\baselineskip}
In fact, a combination of the metric and Killing form provides an inner product on each $V_j^\eta$, such that we can (and do) identify the quotient spaces with the orthogonal complement of $\tilde{V}_j^0$ inside of $V_j^\eta$ (but reflecting this in the notation only adds unnecessary complexity).
In the upcoming paragraphs we will write $\alpha\restrict{\mathrm{(I)}}$ to denote restriction to the subspace specified in case~(I), et cetera.
Let us note at this point that $j=1$ is slightly special, since $V_1^- \cap \tilde{V}_1^0 = \{0\}$.

When we restrict~\eqref{eq:vanishings-logy-1}~-~\eqref{eq:vanishings-logy-4} to either of the subspaces in (I) or (II), all terms on the right hand side vanish.
In particular, \eqref{eq:vanishings-logy-1}~and~\eqref{eq:vanishings-logy-2} reduce to simple eigenvalue equations.
Since $\mathfrak{J}$ acts on $V_j^-$ with eigenvalue $j+1$ (which is strictly larger than $1$), we immediately find $c_{1,k}\restrict{\mathrm{(I)}} = 0 = d_{1,k}\restrict{\mathrm{(I)}}$ for any $j$.
Similarly, since $\mathfrak{J}$ acts on $V_j^+$ with eigenvalue $-j$, we get $c_{1,k}\restrict{\mathrm{(II)}} = 0$, except perhaps when $j = 1$, while $d_{1,k}\restrict{\mathrm{(II)}} = 0$ in any case.
For the $j=1$ component of $c_{1,k}\restrict{\mathrm{(II)}}$ we can rely on the analogue of~\eqref{eq:vanishings-logy-1} at the subleading order $\mathcal{O}\big((\log y)^{k-1}\big)$, which states
\begin{align}
	k c_{1,k} + c_{1,k-1} + \mathfrak{J} c_{1,k-1} &= \sin\beta\cos\beta\ \left( \imath_w  [\phi_\rho^\tau \wedge a_{1,k-1}] - \imath_w [\phi_\rho \wedge a_{1,k-1}] \right)
	\label{eq:vanishings-logy-1-subleading}
\end{align}
Upon restriction to $V_1^+ / \tilde{V}_1^0$ this reduces to $k c_{1,k}\restrict{\mathrm{(II)}} = 0$.

Restriction to the subspaces in (III) and (IV) works out slightly differently, because there the restriction of ${\imath_w [\phi_\rho^\tau \wedge a_{1,k}]}$ doesn't vanish (though we still have vanishing of $\imath_w [\phi_\rho \wedge a_{1,k}]$).
Concentrate on case (III) first, where $J$ has eigenvalue $j+1$.
We find that~\eqref{eq:vanishings-logy-1}~and~\eqref{eq:vanishings-logy-2} determine ${c_{1,k}\restrict{\mathrm{(III)}} = \frac{\sin\beta\cos\beta}{j+2} (\imath_w [\phi_\rho^\tau \wedge a])^{j,-}}$ and ${d_{1,k}\restrict{\mathrm{(III)}} = \frac{\cos^2\beta}{j} (\imath_w [\phi_\rho^\tau \wedge a])^{j,-}}$.
However, when plugging this into~\eqref{eq:vanishings-logy-4}, we get the following consistency condition:
\begin{align}
	j(j+1) \left( \frac{\sin^2\beta\cos\beta}{j+2} - \frac{\cos^3\beta}{j} \right) (\imath_w [\phi_\rho^\tau \wedge a_{1,k}])^{j,-} = - \cos\beta (\imath_w [\phi_\rho^\tau \wedge a_{1,k}])^{j,-}
\end{align}
This can only be satisfied if $\cos 2\beta  = - \frac{2j+3}{(j+1)^2}$ or if $(\imath_w [\phi_\rho^\tau \wedge a_{1,k}])^{j,-} = 0$.
Interestingly, for any $j\geq 2$ there exists a single value of $\beta \in (0,\pi/2)$ for which the first equation is satisfied.
Therefore, whenever $\beta$ is related to one of the spins $j$ in this way, both $(c_{1,k})^{j,-}$ and $(d_{1,k})^{j,-}$ are given as stated above by a multiple of $(\imath_w [\phi_\rho^\tau \wedge a_{1,k}])^{j,-}$.
For generic values of $\beta$, however, we find that $(\imath_w [\phi_\rho^\tau \wedge a_{1,k}])^{j,-}$ vanishes, and thus also $c_{1,k}\restrict{\mathrm{(III)}} = d_{1,k}\restrict{\mathrm{(III)}} = 0$.

For case (IV), recall that $J$ acts with eigenvalue $-j$ on $V_j^+$.
If $j=1$, equation~\eqref{eq:vanishings-logy-1} directly yields ${\imath_w [\phi_\rho^\tau \wedge a_{1,k}]\restrict{\mathrm{(IV)}} = 0}$, such that $d_{1,k}\restrict{\mathrm{(IV)}}= 0$ by~\eqref{eq:vanishings-logy-2} and $c_{1,k}\restrict{\mathrm{(IV)}} = 0$ by~\eqref{eq:vanishings-logy-4}.
Otherwise, for $j\geq 2$, equations~\eqref{eq:vanishings-logy-1}~and~\eqref{eq:vanishings-logy-2} specify $c_{1,k}\restrict{\mathrm{(IV)}}= \frac{\sin\beta\cos\beta}{1-j} \imath_w[\phi_\rho^\tau \wedge a_{1,k}]\restrict{\mathrm{(IV)}}$ and $d_{1,k}\restrict{\mathrm{(IV)}}= -\frac{\cos^2\beta}{1+j} \imath_w[\phi_\rho^\tau \wedge a_{1,k}]\restrict{\mathrm{(IV)}}$.
Plugging these expressions into \eqref{eq:vanishings-logy-4} now leads to the condition $\cos 2\beta = \frac{2j-1}{j^2}$ if $j\geq 2$.
In this condition, we can also allow $j=1$ as input, since the corresponding solution is $\beta = 0$, which we exclude anyway.
If $\beta$ coincides with one of these special values, $(c_{1,k})^{j,+}$ and $(d_{1,k})^{j,+}$ are determined as above in terms of $(\imath_w[\phi_\rho^\tau \wedge a_{1,k}])^{j,+}$.
For generic values of $\beta$, we have $(\imath_w [\phi_\rho^\tau \wedge a_{1,k}])^{j,+}= c_{1,k}\restrict{\mathrm{(IV)}} = d_{1,k}\restrict{\mathrm{(IV)}} = 0$

Case (V) is comparatively simple.
Observe that $\imath_w[\phi_\rho^\tau \wedge a_{1,k}]\restrict{\mathrm{(V)}} = 0$ and that $J$ acts with eigenvalue $1$.
This means that~\eqref{eq:vanishings-logy-2} immediately provides $\imath_w [\phi_\rho \wedge a_{1,k}]\restrict{\mathrm{(V)}} = 0$, such that~\eqref{eq:vanishings-logy-1} states $c_{1,k}\restrict{\mathrm{(V)}} = 0$ and as a consequence~\eqref{eq:vanishings-logy-3} yields $d_{1,k}\restrict{\mathrm{(V)}} = 0$.

Finally, we arrive at case (VI), where neither $\imath_w [\phi_\rho^\tau \wedge a_{1,k}]\restrict{\mathrm{(VI)}}$ nor $\imath_w [\phi_\rho \wedge a_{1,k}]\restrict{\mathrm{(VI)}}$ vanish.
Let us first consider generic values of $\beta$, where we already know that $(\imath_w[\phi_\rho^\tau \wedge a_{1,k}])^{j,\pm} = 0$.
Using ${(\imath_w[\phi_\rho \wedge a_{1,k}])^\tau} = {\imath_w[\phi_\rho^\tau \wedge a_{1,k}]}$ and that acting with $\tau$ twice is the identity map, we can then deduce that $\imath_w[\phi_\rho \wedge a_{1,k}]$ decomposes into a sum of $\tau$-eigenvectors inside of $V_j^0 \cap \tilde{V}_j^0$ with eigenvalues $\pm 1$.
Plugging this decomposition into~\eqref{eq:vanishings-logy-1}~-~\eqref{eq:vanishings-logy-4} then shows that there can only be non-trivial solutions if $j=1$ and $\beta=\pi/2$.
Since we excluded $\beta = \pi/2$ already in the definition of the Nahm pole boundary condition, we find that $(a_{1,k})^{j,w} = (c_{1,k})^{j,0} = (d_{1,k})^{j,0} = 0$ for all $j$.

Still in case $(VI)$, but when $\beta$ is one of the special values $\cos 2\beta = -\frac{2j+3}{(j+1)^2}$ or $\frac{2j-1}{j^2}$, it is no longer assured that $(\imath_w [\phi_\rho^\tau \wedge a_{1,k}])^{j,\pm}$ vanishes.
As a result we can't decompose it into $\pm$-eigenparts of $\tau$ within $V_j^0 \cap \tilde{V}_j^0$.
In this situation equations~\eqref{eq:vanishings-logy-1}~-~\eqref{eq:vanishings-logy-4} instead provide  $(c_{1,k})^{j,0} = - \tan\beta\ \imath_w [\phi_\rho \wedge a_{1,k}]$ and $(d_{1,k})^{j,0} = \imath_w [\phi_\rho \wedge a_{1,k}]$.

By way of an intermediate conclusion, it can thus be stated that for generic $\beta$ and $k\geq 2$, the functions $a_{1,k}$ and $b_{1,k}$ vanish.
Meanwhile, if $\beta$ is one of the (finitely many) special angles determined by $\cos 2\beta = - \frac{2j+3}{(j+1)^2}$ or $\frac{2j-1}{j^2}$, their spin $j$ components are not necessarily zero and are determined -- via $a_{1,k}^\tau = \sin\beta\ c_{1,k} + \cos\beta\ d_{1,k}$ and $b_{1,k} = \cos\beta\ c_{1,k} - \sin\beta d_{1,k}$ -- by the following expressions
\begin{align}
\begin{split}
	c_{1,k} &= \frac{\sin\beta\cos\beta}{j+2} \left( \imath_w [\phi_\rho^\tau \wedge a_{1,k}] \right)^{j,-} - \tan\beta\ \left( \imath_w [\phi_\rho \wedge a_{1,k}] \right)^{j,0} - \frac{\sin\beta\cos\beta}{j-1} \left( \imath_w [\phi_\rho^\tau \wedge a_{1,k}] \right)^{j,+} \\
	d_{1,k} &= \frac{\cos^2\beta}{j} \left( \imath_w [\phi_\rho^\tau \wedge a_{1,k}] \right)^{j,-} + \left( \imath_w [\phi_\rho \wedge a_{1,k}] \right)^{j,0} - \frac{\cos^2\beta}{j+1} \left( \imath_w [\phi_\rho^\tau \wedge a_{1,k}] \right)^{j,+}
\end{split}
\label{eq:vanishings-logy-result1}
\end{align}
The lower order terms $a_{1,k^\prime}$ and $b_{1,k^\prime}$, $0 \leq k^\prime <k$  are then determined by induction.
Let us stress that the possible values of $j\geq 2$ are determined by the (finitely many) spins $j\in J$ that appear in the decomposition of $\ad E$ under the action of $\mathfrak{su}(2)_\mathfrak{s}$.

As an aside, note that all components are determined by the free parameters $\{ (a_{1,k})^{j,w},\ (a_{1,k-1})^{j,w}, \ldots \}$, which describe the $y$-dependence of $\imath_w A$, i.e. the $w^\flat$-component of the gauge connection.
While our current analysis cannot determine the maximal power of $(\log y)^k$, this is in general controlled by the geometry and topology of $E\to M^5$.
For example, if the boundary is of the form $W^4  = \R \times X^3$ and $w$ is the vector field parallel to the real line, then upon dimensional reduction $\imath_w A$ would be identified with the $dy$-component of a one-form $\phi$ that satisfies the Kapustin-Witten equations.
In that situation a well-known vanishing theorem states that $\imath_w A$ vanishes if $A+i\phi$ approaches an irreducible flat connection at $y\to\infty$.
In that way, the regularity of twisted Nahm pole solutions is further controlled by global data, away from the boundary.

The arguments above break down at $k=1$, where in case (II) it is no longer possible to deduce that the $V_1^+ / \tilde{V}_1^0$ part of $c_{1,1}$ vanishes.
This is because for this component we needed to rely on the subleading order of~\eqref{eq:vanishings-Haydys-Witten-equations-rewriting-1}~and~\eqref{eq:vanishings-Haydys-Witten-equations-rewriting-2}.
For $c_{1,1}\restrict{\mathrm{(II)}}$ this means that we have to look at the equations at order $\mathcal{O}(1)$, which receive an additional contribution from $F_\omega$:
\begin{align}
	c_{1,1} + c_{1,0} + \mathfrak{J} c_{1,0}
	&= \phantom{\mathrel{-}} \cos\beta\ \imath_w F_\omega + \imath_w \hodge_4 F_\omega + \sin\beta\cos\beta \left( \imath_w  [\phi_\rho^\tau \wedge a_{1,0}] - \imath_w [\phi_\rho \wedge a_{1,0}] \right) \\
	d_{1,0} - \mathfrak{J} d_{1,0}
	&= - \sin\beta\ \imath_w F_{\omega}^\tau - \cos^2\beta\ \imath_w [\phi_\rho^\tau \wedge a_{1,0}] - \sin^2\beta\ \imath_w  [\phi_\rho \wedge a_{1,0}]
\end{align}
Since the terms that contain $F_\omega$ are located in $V_1^- \oplus V_1^0 \oplus V_1^+$, only these parts of $c_{1,1}$, $c_{1,0}$ and $d_{1,0}$ are affected by curvature contributions.
Note that the special values of $\beta$ are related to spins $j\geq 2$, so there are no additional contributions to the $V_1^- \oplus V_1^0 \oplus V_1^+$ part of the equations.
By restriction of the $\mathcal{O}(1)$ equations to $V_1^-$, $V_1^0$, and $V_1^+$ (remembering that $V_1^- \cap \tilde{V}_1^0 = \{ 0\}$) we thus obtain
\begin{align}
\begin{split}
	c_{1,1} =
		& \ \left( \cos\beta\ \imath_w F_\omega + \imath_w \hodge_4 F_\omega \right)\restrict{V_1^+ / \tilde{V}_1^0} \\[1em]
	c_{1,0} =
		&\phantom{{}+{}} \tfrac{1}{3} \left(\cos\beta\ \imath_w F_\omega + \imath_w \hodge_4 F_\omega \right)^{1,-} \\
		&+ \tfrac{1}{2} \left(\cos\beta\ \imath_w F_\omega + \imath_w \hodge_4 F_\omega + \sin\beta\cos\beta \left( \imath_w  [\phi_\rho^\tau \wedge a_{1,0}] - \imath_w [\phi_\rho \wedge a_{1,0}] \right) \right)^{1,0} \\
		&- \tfrac{1}{2} \left(\cos\beta\ \imath_w F_\omega^\tau + \sin\beta\cos\beta\ \imath_w [\phi_\rho^\tau \wedge a_{1,0}] \right)\restrict{V_1^+ \cap \tilde{V}_1^0}
	 	 +  C^{1,+}
	\\[1em]
	d_{1,0} =
		&\ \left(\sin\beta\ \imath_w F_\omega^\tau \right)^{1,-}
		+ \tfrac{1}{2} \left( \imath_w [\phi_\rho \wedge a_{1,0}] - \cot\beta\ c_{1,0} \right)^{1,0}
		- \tfrac{1}{2} \left( \sin\beta\ \imath_w F_\omega^\tau + \cos^2\beta\ \imath_w [\phi_\rho^\tau \wedge a_{1,0}] \right)^{1,+}
\end{split}
\label{eq:vanishings-logy-result2}
\end{align}
Here, $C^{1,+} \in V_1^+ / \tilde V_1^0$ is some undetermined "integration constant" that cannot be fixed by the current analysis.
In contrast, $\imath_w[\phi_\rho \wedge a_{1,0}]$ and $\imath_w [\phi_\rho^\tau \wedge a_{1,0}]$ are further restricted by the $\mathcal{O}(1)$ versions of~\eqref{eq:vanishings-logy-3}~and~\eqref{eq:vanishings-logy-4} (which remain unchanged) and consequently depend explicitly on $F_\omega$.
Indeed, if the $F_\omega$ contributions to the $\mathcal{O}(1)$ equations vanish, our earlier arguments show that, apart from $c_{1,0} = C^{1,+}$, all coefficient functions vanish.

Finally, if $\beta$ is one of the special angles where $\cos 2\beta = -\frac{2j+3}{(j+1)^2}$ or $\frac{2j-1}{j^2}$, then the results for $c_{1,1}$, $c_{1,0}$ and $d_{1,0}$ differ only by additionally including the terms induced by~\eqref{eq:vanishings-logy-result1}, which only appear for spins $j\neq 1$.

The main take away of this discussion is that, for generic $\beta$ and in absence of curvature contributions, the functions $a_{1,k}$ and $b_{1,k}$ vanish for all $k\geq 1$, while they are proportional to $C^{1,+} \in V_1^+/ \tilde{V}_1^0$ for $k=0$.

\vspace{0.5em}
\underline{$\mathcal{O}(y^1 (\log y)^k )$}\\
For the constraints at order $\mathcal{O}(y (\log y)^k)$ the analysis is very similar.
The main difference is that the equations now incorporate terms of the form $d_\omega c_{1,k}$.
As before, let $k$ be the largest integer for which $(2,k) \in \Delta_0$.
The Haydys-Witten equations \eqref{eq:vanishings-Haydys-Witten-equations-rewriting-1}~-~\eqref{eq:vanishings-Haydys-Witten-equations-rewriting-4} now become
\begin{align}
	\left(2 c_{2,k} + \mathfrak{J} c_{2,k} \right) - \imath_w d_\omega d_{1,k}^\tau - \phantom{(} \imath_w \hodge_4 d_\omega a_{1,k} \phantom{){}^\tau} \
		&= \phantom{\mathrel{-}} \sin\beta\cos\beta\ \left( \imath_w  [\phi_\rho^\tau \wedge a_{2,k}] - \imath_w [\phi_\rho \wedge a_{2,k}] \right)  \\
	\left(2  d_{2,k} - \mathfrak{J} d_{2,k} \right) + \imath_w d_\omega c_{1,k}^\tau  - \left( \imath_w \hodge_4 d_\omega b_{1,k} \right){}^\tau
		&= - \cos^2\beta\ \imath_w [\phi_\rho^\tau \wedge a_{2,k}] - \sin^2\beta\ \imath_w [\phi_\rho \wedge a_{2,k}] \\
	j(j+1) \left( \cos\beta\ c_{2,k} + \sin\beta\ d_{2,k} \right)^{j,0} &= \phantom{\mathrel{-}} \sin\beta\ \imath_w [\phi_\rho \wedge a_{2,k}] \\
	j(j+1) \left( \sin\beta\ c_{2,k} - \cos\beta\ d_{2,k} \right)^{\widetilde{j,0}} &= \phantom{\mathrel{-}} \cos\beta\ \imath_w [\phi_\rho^\tau \wedge a_{2,k}]
\end{align}
We have again arranged this such that the right hand side is contained in $V_j^0 + \tilde{V}_j^0$.
Moreover, we know that for generic values of $\beta$ and any $k\geq 2$, the functions $c_{1,k}$ and $d_{1,k}$ vanish completely, such that the derivatives on the left hand side disappear.
In that situation essentially the same arguments as before show that $(a_{2,k})^{j,w} = c_{2,k} = d_{2,k} = 0$, as long as we additionally exclude the special angles determined by $\cos 2\beta = -\frac{3j^2-4j+3}{j(j+1)^2}$ or $-\frac{3j^2+2j-4}{j^2(j+1)}$.

In carrying out these arguments, we need to rely on the subleading equations at order $\mathcal{O}(y (\log y)^{k-1})$ again.
Here, the subleading equations provide conditions for the $V_2^+/ \tilde{V}_2^0$ part of $c_{2,k}$ and the $V_1^-/\tilde{V}_1^0$ ($= V_1^-$) part of $d_{2,k}$, which are projected out from the equations at order $\mathcal{O}(y (\log y)^k)$.
Specifically, the equation for the $c$-terms at order $\mathcal{O}(y (\log y)^{k-1})$ is
\begin{align}
\begin{multlined}[0.8\textwidth]
	\left( k c_{2,k} + 2 c_{2,k-1} + \mathfrak{J} c_{2,k-1} \right) - \imath_w d_\omega d_{1,k-1}^\tau - \imath_w \hodge_4 d_\omega a_{1,k-1} \\
		= \sin\beta\cos\beta \left( \imath_w  [\phi_\rho^\tau \wedge a_{2,k-1}] - \imath_w [\phi_\rho \wedge a_{2,k-1}] \right)
\end{multlined}
\end{align}
The $V_2^+/\tilde{V}_2^0$ part of this equation simply states $k c_{2,k}\restrict{\mathrm{(II)}} = 0$ for all $k\geq 1$ (for $k=1,2$ notice that $c_{1,k}$ and $a_{1,k}$ are elements of $V_1^- \oplus V_1^0 \oplus V_1^+$).
Similarly, the equation for the $d$ terms at order $\mathcal{O}(y (\log y)^{k-1})$ states
\begin{align}
\begin{multlined}[0.8\textwidth]
	\left(k d_{2,k} + 2 d_{2,k-1} - \mathfrak{J} d_{2,k-1} \right) + \imath_w d_\omega c_{1,k-1}^\tau - (\imath_w \hodge_4 d_\omega b_{1,k-1})^\tau \\
		= - \cos^2\beta\ \imath_w [\phi_\rho \wedge a_{2,k-1}] - \sin^2\beta\ \imath_w [\phi_\rho^\tau \wedge a_{2,k-1}]
\end{multlined}
\end{align}
For $k>2$ the $V_1^-$ part of this equation simply states $(k d_{2,k})^{1,-} = 0$.
For $k=2$ we instead have $(2 d_{2,2}^\tau + \imath_w d_\omega c_{1,1} - \imath_w \hodge_4 d_\omega b_{1,1})^{1,-} = 0$.
Note that $b_{1,1} = \cos\beta\ c_{1,1}$ is an element of $V_1^+/ \tilde{V}_1^0$.
By comparison with the mapping properties~\eqref{eq:vanishings-exterior-derivative-mapping-properties}, we conclude that also $(d_{2,2})^{1,-} = 0$.

The equations at order $\mathcal{O}(y (\log y))$ and $\mathcal{O}(y)$ involve derivatives of $c_{1,1}$, $c_{1,0}$ and $d_{1,0}$, which are all contained in $V_1^- \oplus V_1^0 \oplus V_1^+$, so only these parts of $c_{2,1}$, $d_{2,1}$, $c_{2,0}$ and $d_{2,0}$ may be non-zero.
By restricting the equations to the various possible subspaces (and since $V_1^- \cap \tilde{V}_1^0 = \{0\}$) one obtains
\begin{align}
	c_{2,1} =
		&\phantom{{}+{}} \tfrac{1}{3} \left( \sin\beta\ \imath_w \hodge_4 d_\omega c_{1,1} + \sin\beta\cos\beta\ (\imath_w [\phi_\rho^\tau \wedge a_{2,1}] - \imath_w [\phi_\rho \wedge a_{2,1}]) \right)^{1,0} \\
		&+ \phantom{\tfrac{1}{3}} \left( \sin\beta\ \imath_w \hodge_4 d_\omega c_{1,1} + \sin\beta\cos\beta\ \imath_w [\phi_\rho^\tau \wedge a_{2,1}] \right)^{1,+}
	\\[0.5em]
	d_{2,1} =
	& \phantom{{}+{}} \left( - \imath_w d_\omega c_{1,0}^\tau + (\imath_w \hodge_4 d_\omega b_{1,0})^\tau  \right)^{1,-} \\
	& + \left( - \imath_w d_\omega c_{1,1}^\tau  + (\imath_w \hodge_4 d_\omega b_{1,1})^\tau - \cos^2\beta\ \imath_w [\phi_\rho^\tau \wedge a_{2,1}] - \sin^2\beta\ \imath_w [\phi_\rho \wedge a_{2,1} \right)^{1,0} \\
	& + \tfrac{1}{3} \left( - \imath_w d_\omega c_{1,1}^\tau + (\imath_w \hodge_4 d_\omega b_{1,1})^\tau - \sin^2\beta\ \imath_w [\phi_\rho^\tau \wedge a_{2,1}] \right)^{1,+}
\end{align}
Since the $\mathcal{O}(y \log y)$ terms are non-zero, they appear in the $\mathcal{O}(y)$ equations and, accordingly, in the following result for $c_{2,0}$ an $d_{2,0}$:
\begin{align}
	c_{2,0} =
		& \phantom{{}+{}} \tfrac{1}{4} \left( \imath_w d_\omega d_{1,0}^\tau + \imath_w \hodge_4 d_\omega a_{1,0} \right)^{1,-} \\
		& + \tfrac{1}{3} \left( \imath_w d_\omega d_{1,0}^\tau + \imath_w \hodge_4 d_\omega a_{1,0} + \sin\beta\cos\beta\ \left( \imath_w  [\phi_\rho^\tau \wedge a_{1,0}] - \imath_w [\phi_\rho \wedge a_{1,0}] \right) - c_{2,1} \right)^{1,0} \\
		& + \phantom{\tfrac{1}{3}} \left( \imath_w d_\omega d_{1,0}^\tau + \imath_w \hodge_4 d_\omega a_{1,0} + \sin\beta\cos\beta\ \imath_w  [\phi_\rho^\tau \wedge a_{1,0}] - c_{2,1}  \right)^{1,+}
		+ C^{2,+}
	\\[0.5em]
	d_{2,0} =
		&\phantom{+} D^{1,-}
		+ \left( - \imath_w d_\omega c_{1,0}^\tau  + \left( \imath_w \hodge_4 d_\omega b_{1,0} \right){}^\tau - \cos^2\beta\ \imath_w [\phi_\rho^\tau \wedge a_{2,0}] - \sin^2\beta\ \imath_w [\phi_\rho \wedge a_{2,0}] - c_{2,1} \right)^{1,0} \\
		&+ \tfrac{1}{3} \left( - \imath_w d_\omega c_{1,0}^\tau  + \left( \imath_w \hodge_4 d_\omega b_{1,0} \right){}^\tau - \cos^2\beta\ \imath_w [\phi_\rho^\tau \wedge a_{2,0}] - c_{2,1} \right)^{1,+}
\end{align}
Here, $C^{2,+} \in V_2^+ / \tilde{V}_2^0$ and $D^{1,-} \in V_1^-$ remain undetermined.
While we do not reproduce explicit formulas, it is easy to see that $\imath_w [\phi_\rho \wedge a_{2,k}]$ and $\imath_w [\phi_\rho^\tau \wedge a_{2,k}]$ are fully determined by the bottom two Haydys-Witten equations.

The main conclusion of this part is that for generic $\beta$ and in absence of curvature contributions to the $\mathcal{O}(1)$ equations (such that $a_{1,k} = b_{1,k} = 0$ for $k\geq 1$) all components of $c_{2,k}$ and $d_{2,k}$ vanish, except for the functions $C^{2,+}$ and $D^{1,-}$ that appear at $k=0$.
\end{proof}


\section{Asymptotics of the Boundary Term}
\label{sec:vanishings-asymptotics-of-chi}

We now determine the asymptotic behaviour of $\chi = \chi_1 + \chi_2$ at cylindrical ends of $M^5$, where we recall from \autoref{lem:vanishings-weitzenböck} that
\begin{align}
	\chi_1 &= - 2 \Tr \left( F_A \wedge J^- B \right)\ ,  \\
	\chi_2 &= - 2 \Tr \left( \delta_A^+ J^+ B \wedge J^- B \wedge \eta \right)\ .
\end{align}
In doing so, we assume that we evaluate $\chi$ for a Haydys-Witten solution, which is the situation we are interested in for the Weitzenböck formula.

Since we eventually need to pullback $\chi$ to the boundary, it is important to understand the interplay between the vector subbundles $\ker \eta$ and $T (\del_i M)$, with a particular focus on the properties of $J$.
Thus, let $U = [0,1)_{s_i} \times \del_i M^5$ be a cylindrical end and let $u$ denote the inward-pointing unit normal vector.
This means that near a compact end $u=\del_{s_i}$, while at a non-compact end $u={s_i}\del_{s_i}$.
Assume that $g(u,v) = \cos\theta$ is constant on all of $U$.

If $\theta = 0$, $u$ and $v$ are parallel and $\ker \eta = T \del_i M$.
In this situation $J$ simply corresponds to an almost Hermitian structure on $\del_i M$.
If $\theta \neq 0$, $u$ and $v$ are linearly-independent, non-vanishing vector fields on the tubular neighbourhood and span a regular distribution $\Delta_{(u,v)} \subset TU$ of rank two.
Since $\ker\eta$ has rank four, the two distributions intersect in a line bundle $L \subset TU$.
At each point, $L$ specifies a direction in the $(u,v)$-plane that is perpendicular to $v$.
We denote the generating unit vector field by $v^\perp$ and fix its orientation such that $g(u,v^\perp) = -\sin\theta$, such that it points into the interior.
The vector fields are related by $v=\sin\theta w + \cos\theta u$ and $v^\perp = \cos\theta w - \sin\theta u$.

Starting from $v^\perp$, we can always find a local basis of $\ker \eta$ that interacts nicely with the almost complex structure $J$.
It is given by $\{v^\perp , w_1=J v^\perp, w_2, w_3=J w_2\}$, where $w_2$ is some local section of $T\del_i M$ that is orthogonal to $v$, $v^\perp$, and $J v^\perp$.
This induces an associated basis of $\Omega^2_{v,+}(U)$, given by
\begin{align}
	e_i = \eta^\perp \wedge w_i^\flat + \thalf \epsilon_{ijk} w_j^\flat \wedge w_k^\flat
\end{align}
and that satisfies $J e_1 = + e_1$ and $Je_{2/3} = - e_{2/3}$.

We start our investigation of the asymptotics of $\chi$ with Nahm pole boundaries $\del_{\mathrm{NP}} M$ in \autoref{sec:vanishings-nahm-pole-boundaries}.
Though we only establish \autoref{bigthm:decoupling} in the context of pure Nahm pole boundary conditions, we also include a short discussion of the expected behaviour of $\chi$ near knot boundaries $\del_K M$ in \autoref{sec:vanishings-knot-singularity-boundary}.
Subsequently, we discuss Kapustin-Witten ends with either finite energy or Nahm poles in \autoref{sec:vanishings-nagy-oliveira-asymptotic-ends}.


\subsection{Nahm Pole Boundaries}
\label{sec:vanishings-nahm-pole-boundaries}

Consider a Nahm pole boundary component $\del_{\mathrm{NP}} M^5$.
Let $W^4 \times [0,1)_y$ denote a tubular neighbourhood, where $W^4$ is a complete Riemannian manifold without boundary, and write $\mu_{W^4}$ for the induced volume form.
Write $u = \del_y$ for the inward-pointing unit normal vector field and assume $g(u,v) = \cos \beta$ is constant, with $\beta \in [0,\pi/2)$.
Furthermore, denote by $i_y: W^4 \hookrightarrow W^4 \times [0,1)_y$ inclusion of $W^4$ at $y$.

Recall from \autoref{sec:vanishings-weitzenböck} that there is a local basis $\{v^\perp, w_i\}_{i=1,2,3}$ of $\ker\eta$, where each $w_i$ is parallel to $TW^4$, and associated to it a basis $\{e_i\}_{i=1,2,3}$ of $\Omega^2_{v,+}(W^4 \times [0,1)_y)$ for which $J^- B = B_2 e_2 + B_3 e_3$.
Let $(s, x^i, y)_{i=1,2,3}$ be coordinates that make these into coordinate vector fields.
With that choice we have $v = \sin\beta \del_s + \cos\beta \del_y$, $v^\perp = \cos\beta \del_s - \sin\beta \del_y$, $w = \del_s$, and $w_i = dx^i$.

\begin{proposition}
\label{prop:vanishings-chi-np-asymptotics}
Assume $(A,B)$ is a $\beta$-twisted regular Nahm pole solution of the Haydys-Witten equations.
If $\beta$ is generic, $F_\omega = 0$ and $\nabla^\omega_{s} (C^{1,+})_2 + \nabla^\omega_{1} (C^{1,+})_3 = \nabla^\omega_s (C^{1,+})_3 - \nabla^\omega_1 (C^{1,+})_2 = 0$, then there is a $\delta>0$ and a constant $C$ such that
\begin{align}
	i^\ast_y \chi \sim C y^\delta \mu_{W^4} \quad (y \to 0)
\end{align}
\end{proposition}
\begin{proof}
Consider first the term $\chi_1 = -2 \Tr(F_A \wedge J^- B)$.
Postponing the discussion of the $ds$-components for the moment, let's concentrate on the contributions from $(F_A)_{ij} dx^i\wedge dx^j$.
Observe that $e_{i}$ pulls back to $\cos\beta ds\wedge dx^{i} + \frac{1}{2} \epsilon_{ijk} dx^j\wedge dx^k$.
Fixing an orientation in which $\mu_{W^4} = \sqrt{g} ds \wedge dx^1 \wedge dx^2 \wedge dx^3$, we find that these contributions are given by $\big( (F_A)_{12} B_3 - (F_A)_{13} B_2 \big) \cos\beta \mu_{W^4}$.

Plugging the expansions of \autoref{lem:vanishings-expansion-of-nahm-pole-solutions} and \autoref{cor:vanishings-expansion-of-field-strength} into this formula, we encounter the following terms at orders $y^{-3}$ and $y^{-1}$
\begin{align}
	&y^{-3} i_y^\ast\; \Tr([\phi_\rho^\tau \wedge \phi_\rho^\tau] \wedge J^- \phi_\rho ) \propto \Tr( [\mathfrak{t}_1^\tau, \mathfrak{t}_2^\tau] \mathfrak{t}_3) - \Tr([\mathfrak{t}_1^\tau, \mathfrak{t}_3^\tau] \mathfrak{t}_2) = - 2 \Tr(\mathfrak{t}_2 \mathfrak{t}_3) = 0 \\[0.5em]
	&\begin{multlined}[0.8\textwidth]
		y^{-1} i_y^\ast\; \Tr([\phi_\rho^\tau \wedge \phi_\rho^\tau] \wedge J^- C^{1,+} ) + y^{-1} i_y^\ast\; \Tr([\phi_\rho^\tau \wedge (C^{1,+})^\tau]\wedge J^- \phi_\rho) \\
		\propto \Tr\left( - \mathfrak{t}_2 (C^{1,+})_3 - \mathfrak{t}_3 (C^{1,+})_2 \right) + \Tr\left( \mathfrak{t}_2 (C^{1,+})_3 + \mathfrak{t}_3 (C^{1,+})_2] \right)= 0
	\end{multlined}
\end{align}
Meanwhile, at order $\mathcal{O}(1)$, we have the following contributions
\begin{align}
	i_y^\ast\; \Tr( [\phi_\rho^\tau \wedge \phi_\rho^\tau] \wedge J^- C^{2,+} )  &= 0 &
	i_y^\ast\; \Tr( [\phi_\rho^\tau \wedge \phi_\rho^\tau] \wedge J^- D^{1,-} ) &= 0 \\
	i_y^\ast\; \Tr( [\phi_\rho^\tau \wedge (C^{2,+})^\tau] \wedge J^- \phi_\rho ) &= 0 &
	i_y^\ast\; \Tr( [\phi_\rho^\tau \wedge (D^{1,-})^\tau] \wedge J^- \phi_\rho ) &= 0
\end{align}
On the left we used that the subspaces $V_1,V_2 \subset \mathfrak{g}$ are orthogonal with respect to the trace (keeping in mind that the action of $\phi_\rho \in V_1$ preserves $V_2$), while the equations on the right similarly use $D^{1,-} \in V_1^-$ insofar that $D^{1,-} \propto \phi_\rho$, such that the trace vanishes by the same calculation as at order $\mathcal{O}(y^{-3})$.

According to \autoref{cor:vanishings-expansion-of-field-strength}, the only $ds$-component of $F_A$ can arise from $d_\omega C^{1,+}$, which appears at $\mathcal{O}(1)$.
As we have seen, all other contributions vanish, so we are left with
\begin{align}
	i_y^\ast \chi_1
	=& -2 i_y^\ast\, \Tr( \sin\beta\cos\beta\ d_\omega (C^{1,+})^\tau \wedge J^- \phi_\rho ) \\
	=&
	\begin{aligned}[t]
		&-2 \sin\beta\cos^2\beta\ \mu_{W^4}\ \Tr \left( (\nabla^\omega_s (C^{1,+})_2 + \nabla^\omega_1 (C^{1,+})_3 )  \mathfrak{t}_3 + (\nabla^\omega_s (C^{1,+})_3 - \nabla^\omega_1 (C^{1,+})_2 ) \mathfrak{t}_2 \right) \\
		&-2 \sin\beta\cos^2\beta\ \mu_{W^4}\ \Tr \left(\nabla^\omega_3 (C^{1,+})_1 \mathfrak{t}_2 - \nabla^\omega_2 (C^{1,+})_1 \mathfrak{t}_3 \right) + \mathcal{O}(y^{\delta})
	\end{aligned}
\end{align}
The first line of the result vanishes by assumption, while the remaining terms cancel in combination with $\chi_2$.

Hence, moving on to $\chi_2 = -2 \Tr( \delta_A^+ J^+ B \wedge J^- B \wedge \eta )$.
Note that $i_y^\ast \eta = \sin\beta ds$ and as a result only the $dx^2$- and $dx^3$-components of $\delta_A^+ J^+ B$ can contribute to this expression.
Specifically, we get $i_y^\ast \chi_2 = 2 \Tr\big( (\delta_A^+ J^+ B)_2 B_2 +  (\delta_A^+ J^+ B)_3 B_3 \big) \sin\theta \mu_{W^4}$.
Upon plugging in \autoref{lem:vanishings-expansion-of-nahm-pole-solutions} and \autoref{cor:vanishings-expansion-of-field-strength} one encounters essentially the same expressions as for $\chi_1$.
The only non-zero terms are again those that contain derivatives of $C^{1,+}$, which are now given by
\begin{align}
	i_y^\ast \chi_2 = 2 \sin\beta\cos^2\beta \mu_{W^4}\ \Tr\left( \nabla^\omega_3 (C^{1,+})_1 \mathfrak{t}_2 - \nabla^\omega_2 (C^{1,+})_1 \mathfrak{t}_3  \right) + \mathcal{O}(y^{\delta})
\end{align}
These cancel the remnants of $i_y^\ast \chi_1$, which concludes the proof.
\end{proof}


\subsection{Knot boundaries}
\label{sec:vanishings-knot-singularity-boundary}

Consider a boundary face of type $\del_K M$ and remember that it arises by blowup of a knot surface $\Sigma_K$.
Let $\Sigma_K \times H^2 \times [0,1)_R$ denote a tubular neighbourhood, where $H^2$ is the two-dimensional hemisphere.
Moreover, assume that the glancing angle $\theta$ between $v$ and $\Sigma_K$ is constant.
If $\theta \neq \pi/2$, there is a non-vanishing vector field $u$ parallel to $\Sigma_K$ such that $g(u,v) = \sin\theta$.

In analyzing the behavior of the function $\chi$ near $\del_K M$, we encounter some difficulties.
The polyhomogeneous expansion of $(a,b)$ in this case starts at $y^0$.
While the leading order of $\chi$ remains essentially unchanged, given by straightforward analogues of the $y^{-3}$-terms from earlier, the additional $y^0$-modes of $a$ and $b$ lead to new contributions, beginning at $\mathcal{O}(y^{-2})$.

To determine these contributions, we need to repeat the analysis of the polyhomogeneous expansion near $\del_K M$, but now based on the knot singularity model solutions.
In this case, one considers model solutions $(A^{\lambda,\theta}, B^{\lambda,\theta})$ that are solutions of a $\theta$-twisted version of the extended Bogomolny equations (TEBE).
Unfortunately, unlike the Nahm pole solutions, these knot singularity solutions are to the most part known only implicitly.
For $G=SU(N)$ and $\theta = \pi/2$ the solutions have been described in \cite{Mikhaylov2012}, while the existence of solutions for $G=SU(2)$ and general $\theta\in(0,\pi/2]$ is described by a continuation argument due to \cite{Dimakis2022}.

For the specific case with $G=SU(2)$ and $\theta=\pi/2$, the solutions have been described in \cite{Witten2011} and can be given explicitly.
To do so, introduce coordinates $(s,t, \psi, \vartheta, R)$ on $\Sigma_K \times H^2 \times [0,1)_R$, $(s,t)$ are local coordinates on the surface $\Sigma_K$ that we will also collectively refer to as $z$, while $R\in[0,1)$, $\psi \in [0,\pi/2]$ and $\vartheta \in [0,2\pi]$ are global coordinates on the filled hemisphere.
Let $(\mathfrak{t}_i)_{i=1,2,3}$ denote a standard basis of $\mathfrak{su}(2)$ and view $\mathfrak{t}_1$ as the generator of a fixed Cartan subalgebra.
Introduce the $\mathfrak{sl}(2,\mathbb{C})$-valued function $\varphi = \phi_2 - i \phi_3$ as a convenient combination of the components $\phi_2$ and $\phi_3$ of $B = \sum_1^3 \phi_i e_i$.
Similarly, denote by $E = \mathfrak{t}_2 - i\mathfrak{t}_3$, $H=\mathfrak{t}_1$, and $F = \mathfrak{t}_2 + i\mathfrak{t}_3$ the elements of an $\mathfrak{sl}(2,\mathbb{C})$-triple $(E,H,F)$.
The knot singularity solutions with charge $\lambda \in \mathbb{Z}$ are given by
\begin{align}\label{eq:background-Nahm-Pole-knot-singularity-model-solutions}
\begin{split}
	A_\vartheta &= - (\lambda+1) \cos^2 \psi \frac{(1+\cos \psi)^{\lambda} - (1-\cos\psi)^{\lambda}}{(1+\cos\psi)^{\lambda+1}-(1-\cos\psi)^{\lambda+1}}\ H \\[1ex]
	\phi_1 &= -\frac{\lambda+1}{R} \frac{(1+\cos\psi)^{\lambda+1} + (1-\cos\psi)^{\lambda+1}}{(1+\cos\psi)^{\lambda+1} - (1-\cos\psi)^{\lambda+1}}\ H \\[1ex]
	\varphi &= \frac{(\lambda+1)}{R} \frac{\sin^\lambda \psi\ \exp(i\lambda\vartheta) }{(1+\cos\psi)^{\lambda+1} - (1-\cos\psi)^{\lambda+1}}\ E \\[1ex]
	A_s &= A_t = A_R = A_\psi = 0
\end{split}
\end{align}
We refer to these as the untwisted knot singularity model solutions.

The higher orders in a polyhomogeneous expansion of $(A,B)$ around these model solutions was previously investigated by He \cite{He2018a, He2019a}.
Originally this was for Nahm pole solutions of the $\theta=\pi/2$ Kapustin-Witten equations, i.e. in four dimensions.
However, the $\theta=\pi/2$ Kapustin-Witten equations are equivalent to a dimensional reduction of the Haydys-Witten equations along a direction perpendicular to $v$, and along which $K$ extends to $\Sigma_K$.
After an inconsequential reinterpretation of field components, the polyhomogeneous expansion carries over to the Haydys-Witten equations with knot singularity.
Indeed, in that context the glancing angle between $v$ and $\Sigma_K$ is $\theta = \pi/2$, such that the untwisted knot singularity solutions are natural boundary conditions.
\begin{theorem}[{\cite{He2018a, He2019a}}]
Let $(A,B)$ be a solution of the Haydys-Witten equations that satisfies Nahm pole boundary conditions with a knot singularity of weight $\lambda$ and with glancing angle $\theta=\pi/2$.
Correspondingly, write $A = A^{\lambda, \pi/2} + a$ and $B=B^{\lambda,\pi/2} + b$ with $a,b \in \mathcal{O}(R^{-1+\epsilon} s^{-1+\epsilon})$.
Then $a$ and $b$ are polyhomogeneous in $R$ and $s$, with non-negative exponents.
In particular they satisfy the estimates
\begin{align}
	\abs{\nabla^\ell_z \nabla^m_R \nabla^n_{\cos \psi}\ a }_{\mathcal{C}^0} \leq C_{\ell,m,n} R^{-\epsilon-m} (\cos\psi)^{2-\epsilon-n} \ , \quad
	\abs{\nabla^\ell_z \nabla^m_R \nabla^n_{\cos \psi}\ b }_{\mathcal{C}^0} \leq C_{\ell,m,n} R^{-\epsilon-m} (\cos\psi)^{1-\epsilon-n}
\end{align}
for any $\epsilon>0$ and $\ell,m,n \in \mathbb{N}$.
\end{theorem}

We can use this to determine the leading order of $\chi$ in the case of $\theta=\pi/2$ knot singularities along $\del_K M$.
For this, let $i_R: \Sigma_K \times H^2 \hookrightarrow \Sigma_K \times H^2 \times [0,1)_R$ denote inclusion at radius $R$ and write $\mu_{\Sigma_K \times H^2_R}$ for the pullback of the volume form under $i_R$, i.e. to the hemisphere (or cylinder) of radius $R$.
First, since the $dR$-components of $F_A$ drop out under pullback and the remaining components contain at most one derivative $\nabla_{\cos\psi}$, only components of $F_A$ of order $\mathcal{O}(R^{-\epsilon} (\cos\psi)^{1-\epsilon})$ contribute to $\chi_1$.
Furthermore, $J^- B = \mathcal{O}(R^{-1}(\cos\psi)^{-1})$, such that $\chi_1 = -2 \Tr(F_A \wedge J^- B) \sim C R^{-1-\epsilon} (\cos\psi)^{-\epsilon} ) \mu_{\Sigma_K \times H^2_R}$
Second, $\delta_A^+ J^+ B$ does not contain $\nabla_R$ and the leading order contributions of the product $J^+ B \wedge J^-B$ at $R^{-2}(\cos\psi)^{-2}$ vanish by construction of the underlying Nahm pole ($\Tr(HE) = 0$).
We conclude that also $\chi_2 = -2 \Tr( \delta_A^+ J^+ B \wedge J^- B \wedge \eta ) \sim C R^{-1-\epsilon} (\cos\psi)^{-\epsilon} ) \mu_{\Sigma_K \times H^2_R}$.
In fact, when the assumptions of \autoref{sec:vanishings-nahm-pole-boundaries} are satisfied, the exponent of $\cos\psi$ can be improved to $(\cos\psi)^\delta$, for some $\delta>0$.
This follows from the fact that $R\cos\psi = y$ and for fixed $R\neq 0$, the expansion in $\cos\psi$ must be consistent with the expansion in $y^\delta$ given in \autoref{prop:vanishings-chi-np-asymptotics}.

All in all, if $(A,B)$ are Haydys-Witten solutions that satisfy (untwisted) Nahm pole boundary conditions near $\del_K M$, then $\chi$ is of order $R^{-1-\epsilon} (\cos\psi)^{\delta}$, for some $\delta> 0$.
It is natural to expect that for general $\theta\in(0,\pi/2]$ the behavior of the polyhomogeneous expansions of $A$ and $B$ near $\del_K M$ follows a similar pattern to the one we observed for the pure $\beta$-twisted Nahm pole solutions at $\del_{\mathrm{NP}} M$.
Specifically, if $A= A^{\lambda, \theta} + a$ and $B=B^{\lambda,\theta} + b$ with respect to the twisted model solutions, it is still true that $a$ and $b$ admit polyhomogeneous expansions of order $\mathcal{O}(R^{-\epsilon} (\cos\psi)^{-\epsilon})$.
In analogy to the cancellations in the case of Nahm pole boundary conditions, we then expect that $a$ and $b$ are restricted in a way that makes any contributions to $\chi$ below $R^{-1-\epsilon}(\cos\psi)^{\delta}$ vanish.

In any case, as we will see below, the following asymptotic behaviour of $\chi$ near a knot boundary $\del_K M$ is sufficient to extend \autoref{bigthm:decoupling} to situations with knot singularities:
\begin{align}
	\chi \sim C R^{-2+\delta} (\cos\psi)^{\delta} \mu_{\Sigma_K \times H^2_R} \quad (R \to 0)
	\label{eq:vanishings-chi-knot-asymptotics}
\end{align}
It's likely that such a result only holds under certain conditions on the topology and geometry of $E\to M^5$ and $\Sigma_K \subset \del_{\mathrm{NP}} M$.


\subsection{Kapustin-Witten ends}
\label{sec:vanishings-nagy-oliveira-asymptotic-ends}

We now consider Kapustin-Witten ends $\del_{\mathrm{KW}} M^5$ and $\del_{\mathrm{NP-KW}} M^5$.
As explained in \autoref{sec:vanishings-setting}, we take this to mean non-compact ends at which fields converge to a Kapustin-Witten solution with either finite energy or with Nahm pole at a corner of $M^5$, respectively.
We let $[0,1)_s \times W^4$ be a tubular neighbourhood and denote by $i_s: W^4 \hookrightarrow [0,1)_s \times W^4$ inclusion of $W^4$ at $s$.
Write $u = s\del_s$ for the inward-pointing unit normal vector field and assume $g(u,v) = \cos \theta$ is constant.
As usual, if $\theta\neq 0$, let $w$ be the non-vanishing unit vector field on $W^4$ with respect to which $v= \cos\theta u + \sin\theta w$.

We begin with a general result about the pullback of $\chi$ in dependence of the limiting field configuration.
Hence, assume $(A,B)$ approaches a solution of the $\theta$-Kapustin-Witten equations $(\tilde{A}, \phi)$ on $W^4$.
In a local basis $(s, t, x^i)_{i=1,2,3}$ with $w = \del_t$, the pullback of $\chi = 2\Tr( F_A \wedge J^- B + \delta_A^+ J^+ B \wedge J^- B \wedge \eta )$ is then given by
\begin{align}
	\thalf \lim_{s \to 0} i_s^\ast \chi 
	= \cos\theta \Tr\big( (F_{\tilde A}^+)_{t2} \phi_2 + (F_{\tilde A}^+)_{t3} \phi_3 \big) \mu_{W^4} 
	+ \sin\theta \Tr\big( \nabla^{\tilde A}_2 \phi_1\  \phi_3 - \nabla^{\tilde A}_3 \phi_1\ \phi_2  \big) \mu_{W^4}
	\label{eq:vanishings-kw-asymptotics}
\end{align}
Here, $F_{\tilde A}^+$ denotes the self-dual part with respect to the four-dimensional Hodge operator $\hodge_{W^4}$.

\paragraph{Finite Energy Solutions.}
Kapustin and Witten observed that Kapustin-Witten solutions on closed manifolds are highly restricted \cite{Kapustin2007}.
A similar result was established by Nagy and Oliveira for finite energy solutions on ALE and ALF gravitational instantons \cite{Nagy2021}.
The energy in question, usually called Kapustin-Witten energy, is defined by the functional
\begin{align}
	E_{\mathrm{KW}} = \int_{W^4} \norm{F_{\tilde A}}^2 + \norm{\nabla^{\tilde A} \phi}^2 + \norm{[\phi\wedge\phi]}^2 
\end{align}
Nagy and Oliveira observed that for finite energy solutions the norm of the Higgs field is bounded, which yields surprisingly strong restrictions for solutions on $\R^4$ and $S^1 \times \R^3$, when combined with a result by Taubes \cite{Taubes2017a}.
This has recently been generalized to arbitrary ALE and ALF gravitational instantons, as was already conjectured by Nagy and Oliveira.
The situation is summarized by the following two theorems.
\begin{theorem*}[\cite{Kapustin2007, Gagliardo2014}]
\label{thm:background-Kapustin-Witten-vanishing}
Let $E\to W^4$ be an $SU(2)$ principal bundle over a compact manifold without boundary.
Assume $(A,\phi)$ satisfies the $\theta$-Kapustin-Witten equations with $\theta \in (0,\pi)$.
If $E\to W^4$ has non-zero Pontryagin number, then $A$ and $\phi$ are identically zero. 
Otherwise $A+i\phi$ is a flat $PSL(2,\mathbb{C})$ connection; equivalently $F_A = [\phi\wedge\phi]$ and $\nabla^A \phi = 0$.
\end{theorem*}
\begin{theorem*}[\cite{Bleher2023a}] \label{thm:vanishings-nagy-oliveira-conjecture}
Let $(A,\phi)$ be a finite energy solution of the $\theta$-Kapustin-Witten equations with $\theta\neq 0 \pmod{\pi}$ on an ALE or ALF gravitational instanton and let $G=SU(2)$.
Then $A$ is flat, $\phi$ is $\nabla^A$-parallel, and $[\phi \wedge \phi] = 0$.
\end{theorem*}

In combination with equation~\eqref{eq:vanishings-kw-asymptotics}, we immediately arrive at the following result.
\begin{proposition} \label{prop:vanishings-chi-kw-asymptotics}
Let $G=SU(2)$ and $\theta\nequiv 0 \pmod{\pi}$.
Assume $(A,B)$ approaches a finite energy solution of the $\theta$-Kapustin-Witten equations on $W^4$ as $s \to 0$.
If $W^4$ is an ALE or ALF gravitational instanton, or if $W^4$ is compact and either
\begin{enumerate*}
\item $E\to W^4$ has non-zero Pontryagin number or
\item $\theta=\pi/2$,
\end{enumerate*}
then $\lim_{s\to 0} i_s^\ast\chi = 0$.
\end{proposition}
\begin{proof}
If $W^4$ is an ALE or ALF gravitational instanton, both terms in~\eqref{eq:vanishings-kw-asymptotics} vanish directly.
The first since $\tilde A$ is flat and the second since $\phi$ is $\nabla^{\tilde A}$ parallel.

If $W^4$ is compact, $A$ and $\phi$ can only be non-zero if the Pontryagin number of $E\to W^4$ is zero and, furthermore, in that case $F_{\tilde A} = [\phi\wedge\phi]$ and $\nabla^{\tilde A} \phi = 0$.
It follows that $\lim_{s\to 0} i_s^\ast \chi = 2 \cos\theta \Tr( [\phi_1,\phi_2] \phi_3) \mu_{W^4}$, which vanishes if $\theta = \pi/2$.
\end{proof}

We make two observations.
First, if $\theta \equiv 0 \pmod{\pi}$, the natural boundary condition for $(A,B)$ are finite energy solutions $(\tilde A, B, C)$ of the Vafa-Witten equations on $W^4$.
On $\R^4$ the Vafa-Witten and $\theta=0$ Kapustin-Witten equations are equivalent by an identification of $\phi = C dx^0 + B_{0i} dx^i$.
According to the result by Taubes mentioned earlier, any solution with bounded $\norm{\phi}$ satisfies $\sigma(B,B) = [B,C] = 0$ and $\nabla^A B = \nabla^A C= 0$ \cite{Taubes2017a}.

It is not currently known if Vafa-Witten solutions on ALX spaces have a similar property.
However, the Vafa-Witten equations are still closely related to the $\theta=0$ version of the Kapustin-Witten equations, and Taubes' result has been established for the latter in \cite{Bleher2023a}.
One might thus expect that, at least in certain situations, the vanishing results carry over to Vafa-Witten solutions.
Whenever this is the case, the Vafa-Witten equations $F_{\tilde A}^+ = \sigma(B,B) + [B,C]$ reduce to anti-self-dual equations for $\tilde A$, in which case $\lim_{s\to 0} i_s^\ast \chi = F_{\tilde A}^+ \wedge i_s^\ast J^- B = 0$, as we found above.

Second, the fact that $\chi$ converges to zero holds at two ends of a spectrum of asymptotic volume growth of $W^4$.
On the one hand, in the case of compact manifolds with bounded volume or equivalently asymptotic volume growth of order $r^0$, and on the other hand, on ALF and ALE manifolds with asymptotic volume growth of order $r^3$ and $r^4$, respectively.
For ALG and ALH gravitational instantons, the proof strategy of Nagy and Oliveira result doesn't work, because they rely on the existence of a positive Green's function for the Laplacian.
The approach to Taubes' dichotomy fails for ALH manifolds for much the same reason.

To the best of the authors knowledge it is currently not known if an analogue of the results of Nagy and Oliveira should be expected to be true or false on ALG and ALH gravitational instantons.
However, from the physics perspective there is no obvious reason to single out intermediate volume growth like that.
On the contrary, in Witten's approach to Khovanov homology, where one considers $M^5 = \R_s \times X^3 \times \R^+_y$, it is natural to expect simplifications whenever $X^3$ has additional substructure.

For example, one expects that Khovanov homology arises for $X^3=S^3$, in which case $M^5$ has an ALH end $W^4 = \R_s \times S^3$ at $y\to\infty$.
Also, Gaiotto and Witten recovered the Jones polynomial purely by adiabatically braiding solutions of the EBE under the assumption that $X^3 = \R_t \times \Sigma^2$, in which case one encounters an ALG manifold at $y \to \infty$.
It thus seems plausible to postulate that the results of this section hold, more generally, whenever $W^4$ is a complete Ricci-flat Riemannian manifold with sectional curvature bounded from below and $(A,B)$ approaches a finite energy solution of the $\theta$-Kapustin-Witten equations ($\theta \neq 0, \pi$) or the Vafa-Witten equations ($\theta =0$).

\paragraph{Nahm Poles at Corners.}
We now discuss the behaviour of $\chi$ at non-compact ends of the class $\del_{\mathrm{NP-KW}} M^5$.
As before, we denote by $[0,1)_s \times W^4$ a tubular neighbourhood.
In this situation we still demand that $(A,B)$ converges to a $\theta$-Kapustin-Witten solution $(\tilde A, \phi)$ on $W^4$ as $s\to 0$.
But in contrast to earlier, we now assume that $(\tilde A, \phi)$ satisfies $\beta$-twisted Nahm pole boundary conditions at an adjacent corner of $M^5$ -- or equivalently at a boundary of $W^4$.
Observe that $(\tilde A, \phi)$ cannot have finite energy if it exhibits a Nahm pole, such that \autoref{prop:vanishings-chi-kw-asymptotics} doesn't apply.

In the absence of a finite energy condition, we now also have to specify boundary conditions at non-compact ends $[0,1)_{s^\prime} \times X^3$ of $W^4$.
We demand that $\tilde A + i\phi$ converges to a flat $G_{\mathbb{C}}$ connection on $X^3$ as $s^\prime \to 0$.
From now on we refer to these configurations simply as Nahm pole solutions of the $\theta$-Kapustin-Witten equations.

\begin{remark}
Note that a non-compact end $[0,1)_{s^\prime} \times X^3$ of $W^4$ corresponds to a "corner at infinity" $[0,1)_{s} \times [0,1)_{s^\prime} \times X^3$ of $M^5$.
This corner is adjacent to two non-compact ends of $M^5$, at which $(A,B)$ converges to a corresponding solution of the $\theta$- or $\theta^\prime$-Kapustin-Witten equation, respectively.
The two associated asymptotic boundary conditions, which demand that both $\tilde A + i \phi$ and $\tilde{A}^\prime + i \phi^\prime$ are flat connections on $X^3$, have to be consistent with the fact that both Kapustin-Witten solutions arise from the common five-dimensional fields $(A,B)$.
Put differently: If we view $B$ as a one-form by the usual isomorphism, the pullback of $A+iB$ converges to a flat connection on $X^3$ as $s,s^\prime \to 0$.
\end{remark}

We similarly need to ensure that the boundary conditions are compatible at corners $[0,1)_s \times X^3 \times [0,1)_y$ that are adjacent to a $\theta$-Kapustin-Witten end as $s\to 0$ and a $\beta$-Nahm pole as $y\to 0$.
This is only the case if $\beta= \pi/2 - \theta$, since otherwise the $\beta$-twisted Nahm pole model solutions are not solutions of the $\theta$-Kapustin-Witten equations.

\begin{proposition} \label{prop:vanishings-chi-npkw-asymptotics}
Assume $(A,B)$ approaches a solution of the $\theta$-Kapustin-Witten equations on $W^4$.
Then 
\begin{align}
	\lim_{s\to 0} i_s^\ast \chi 
	&= \tfrac{2}{3} i_0^\ast \Tr\left( \sigma(B,B) \wedge B \right)
		+ 2\tfrac{\sin\theta}{\cos^2\theta}\ i_0^\ast d \Tr\left( \imath_w (J^+ B) \wedge J^- B \right)
		\ .
\end{align}
and the expression is exact if $\theta=\pi/2$.
\end{proposition}
\begin{proof}
This is a rewriting of equation~\eqref{eq:vanishings-kw-asymptotics}.
Since $(\tilde A, \phi)$ satisfy the $\theta$-Kapustin-Witten equations, we can replace the field strength by $F_{\tilde A} = \frac{1}{2} [\phi\wedge\phi] - \cot\theta d_{\tilde A} \phi + \csc\theta \hodge_4 d_{\tilde A} \phi$.
After a short calculation (and slightly miraculous cancellations), we find
\begin{align}
	\lim_{s\to 0} i_s^\ast \chi = 2 \cos \theta \Tr( [\phi_1, \phi_2] \phi_3 ) \mu_{W^4} + \sin\theta\; \Tr(\nabla^A_2 (\phi_1\phi_3) - \nabla^A_3(\phi_1\phi_2) ) \mu_{W^4} \ ,
	\label{eq:vanishings-chi-npkw-boundary-in-components}
\end{align}
which is a local representation of the expression above and shows that the right hand side is exact if $\cos \theta = 0$
\end{proof}


\section{Vanishing of the Boundary Term}
\label{sec:vanishings-proof-main-result}

We can now show that the contributions from the exact term in the Weitzenböck formula of \autoref{eq:vanishings-weitzenböck} vanishes when the various conditions we have encountered in \autoref{sec:vanishings-polyhomogeneous-expansion-of-nahm-pole-solutions}~and~\autoref{sec:vanishings-asymptotics-of-chi} are satisfied.
In summary we make the following assumptions:
\begin{enumerate}[label=(A\arabic*)]
\item At $\del_{\mathrm{NP}} M^5$ the fields satisfy regular $\beta$-twisted Nahm pole boundary conditions for some generic~$\beta$.
Writing $A = y^{-1} \sin\beta \phi_\rho^\tau + \omega + a$ and $B= y^{-1} \cos\beta \phi_\rho + b$, assume that $F_\omega = 0$ and that $\nabla^\omega_s b_2 + \nabla^\omega_1 b_3 = \mathcal{O}(y^2)$ and $\nabla^\omega_s b_3 - \nabla^\omega_1 b_2 = \mathcal{O}(y^{2})$ (cf. \autoref{sec:vanishings-nahm-pole-boundaries}).

\item At $\del_{K} M^5$ the fields are asymptotic to knot singularity models and there is some $\delta>0$ such that $i^\ast_R \chi \sim C R^{-2+\delta} (\cos\psi)^{\delta} \mu_{\Sigma_K \times H^2_R}$ as $R\to 0$ (cf. \autoref{sec:vanishings-knot-singularity-boundary}).

\item At $\del_{\mathrm{KW}} M^5$ the fields approach a finite energy solution of the $\theta$-Kapustin-Witten equations.
The boundary face $\del_{\mathrm{KW}} M^5$ is either
\begin{enumerate*}
	\item an ALE or ALF gravitational instanton, 
	\item a compact manifold on which the bundle $E\to \del_{\mathrm{KW}} M^5$ has non-zero Pontryagin number, or 
	\item a compact manifold with incidence angle $\theta=\pi/2$
\end{enumerate*}
(cf. \autoref{sec:vanishings-nagy-oliveira-asymptotic-ends}).

\item At $\del_{\mathrm{NP-KW}} M^5$ the incidence angle is $\theta=\pi/2$ and the fields approach Kapustin-Witten solutions with Nahm poles at boundaries.
Moreover, at non-compact cylindrical ends of $\del_{\mathrm{NP-KW}}M^5$, the combination $A+i B$ converges to a flat $G_{\mathbb{C}}$ connection that satisfies $J^- \sigma(B,B) = 0$.
(cf. \autoref{sec:vanishings-nagy-oliveira-asymptotic-ends}).
\end{enumerate}

\begin{theorem} \label{thm:vanishings-decoupling}
Let $G=SU(2)$, $M^5$ a manifold with poly-cylindrical ends, $v$ a non-vanishing vector field that approaches ends at a constant angle, and $J$ an almost Hermitian structure on $\ker \eta$.
Assume $\HW[v](A,B)=0$ and that \textnormal{(A1)}~-~\textnormal{(A4)} are satisfied, then $\dHW[v,J](A,B)=0$.
\end{theorem}

\begin{proof}
Our starting point is a regularized version of the Weitzenböck formula of \autoref{lem:vanishings-weitzenböck}, obtained by restricting the domain of integration to the compact submanifold with corner $M^5_\epsilon$ introduced in \autoref{sec:vanishings-setting} and taking $\epsilon\to 0$.
\begin{align}
	\int_{M^5} \norm{\HW[v](A,B)}^2 = \int_{M^5} \norm{\dHW[v,J](A,B)}^2 + \lim_{\epsilon \to 0} \int_{M^5_\epsilon} d \chi
\end{align}
According to Stokes' theorem, the contributions of the exact term are now determined by 
\begin{align}
	\lim_{\epsilon \to 0} \int_{M_\epsilon} d \chi = \sum_{i \in I} \lim_{\epsilon \to 0} \int_{\del_i M_\epsilon} \chi \ .
\end{align}
We address the integrals for each of the four boundary classes independently.

\paragraph{Nahm Pole Boundaries:}
Let $W^4 \times [0,1)_y$ be a tubular neighbourhood of a Nahm pole boundary $\del_{\text{NP}} M^5$.
The boundary face $\del_{NP}M_\epsilon$ is a subset of the $\epsilon$-displacement $\{\epsilon\}\times W^4 \hookrightarrow [0,1)_s \times W^4$.
We know from \autoref{prop:vanishings-chi-np-asymptotics} that there is some $\delta > 0$ such that $i^\ast_y \chi \sim C y^\delta \mu_{W^4}$ as $y \to 0$.
This provides the estimates
\begin{align}
	\lim_{\epsilon \to 0} \abs{ \int_{\del_{\text{NP}} M_\epsilon} i_\epsilon^\ast  \chi }
	\leq \lim_{\epsilon \to 0} \int_{\del_{\text{NP}} M_\epsilon} \abs{ i_\epsilon^\ast \chi }
	\leq \lim_{\epsilon \to 0} \int_{W^4 \times \{\epsilon\}} \abs{C} \epsilon^{\delta} \mu_{W^4}
	= 0
\end{align}
We have used that the leading order of the pullback of $\chi$ extends unchanged to all of $W^4$, due to compatibility of boundary conditions at corners.

\paragraph{Knot Boundaries:}
Let $\Sigma_K \times H^2_{\psi,\vartheta} \times [0,1)_R$ be a tubular neighbourhood of a knot boundary $\del_K M^5$, where $H^2_{\psi,\vartheta}$ denotes the two-dimensional hemisphere, parametrized by $\psi \in [0,\pi/2]$ and $\vartheta \in [0,2\pi]$.
The associated boundary of the regularized manifold $\del_K M^5_\epsilon$ is contained in the $\epsilon$-displacement $\Sigma_K \times H^2 \times \{\epsilon\} = \Sigma_K \times H^2_\epsilon$.
The pullback of the volume form to the hemisphere of radius $R=\epsilon$ is given by $\mu_{\Sigma_K \times H^2_\epsilon} = \epsilon^2 d\psi d\vartheta \mu_{\Sigma_K}$.
Assuming explicitly that $i_R^\ast \chi \sim c R^{-2+\delta} (\cos\psi)^{\delta} \mu_{\Sigma_K \times H^2_R}$ as $R\to 0$, we find
\begin{align}
	\lim_{\epsilon \to 0} \abs{\int_{\del_{K} M^5_\epsilon } i_\epsilon^\ast  \chi}
	\leq \lim_{\epsilon \to 0} \int_{\Sigma_K\times H^2 \times \{\epsilon\}} \abs{i_\epsilon^\ast \chi}
	\leq \lim_{\epsilon \to 0} \int_{\Sigma_K} \int_{H^2} \abs{C} \epsilon^{-2+\delta} \epsilon^2 d\psi d\vartheta \mu_{\Sigma_K}
	= 0
\end{align}
where we have used the Fubini-Tonelli theorem to split off integration along $\Sigma_K$.
In extending the integral from $\del_K M_\epsilon$ to all of $\Sigma_K \times H \times \{\epsilon\}$, we have used the compatibility of knot singularities with the pure Nahm pole boundary conditions (which are the only boundary conditions that are adjacent to knot singularities) at the corner $\cos\psi\to 0$.

\paragraph{Kapustin-Witten Ends:}
Let $[0,1)_s\times W^4$ be a tubular neighbourhood of a Kapustin-Witten end $\del_{\text{KW}} M^5$.
\autoref{prop:vanishings-chi-kw-asymptotics} states that $\lim_{s\to 0} i_s^\ast \chi = 0$.
Since $\del_{\mathrm{KW}} M_\epsilon \subset W^4$ and $W^4$ is ALE or ALF, its volume grows asymptotically at most with $\epsilon^{-4}$.
Looking back at~\eqref{eq:vanishings-kw-asymptotics}, the rate of decay of $\chi$ is determined by how fast $F_A$ and $\nabla^A \phi$ approach zero as $s\to 0$.
Since the Haydys-Witten equations represent flow equations of the Kapustin-Witten equations, a typical solution is expected to decay exponentially fast towards the stationary solution.
We conclude that in the limit $\epsilon \to 0$:
\begin{align}
	\lim_{\epsilon\to 0} \abs{\int_{\del_{\text{KW}} M^5_\epsilon} \chi }
	\leq \lim_{\epsilon \to 0} \int_{\del_{\text{KW}} M^5_\epsilon} \abs{i_\epsilon^\ast \chi }
	= 0
\end{align}

\paragraph{Kapustin-Witten Ends with Nahm Poles:}
Let $[0,1)_s\times W^4$ be a non-compact end of $M^5$ and assume that $g(\del_s, v) = 0$, i.e. $\theta = \pi/2$.
As before, denote by $i_s : W^4 \hookrightarrow [0,1)_s \times W^4$ inclusion of $W^4$ at $s$.

Assume $(A,B)$ converges to a Kapustin-Witten solution $(\hat A,\phi)$ that exhibits a Nahm pole at a boundary of $W^4$.
Since $\theta=\pi/2$, \autoref{prop:vanishings-chi-npkw-asymptotics} states that $\lim_{s\to 0} i_s^\ast \chi = d \omega$, where $\omega = \Tr\left( \imath_w (J^+ B) \wedge J^- B \right)$.
It follows that
\begin{align}
	\lim_{s\to 0} \int_{\del_{\text{NP-KW}} M_\epsilon} i_s^\ast \chi = \int_{W^4} d\omega
\end{align}
and it remains to determine the integral of $d\omega$ over $W^4$.

For this, let $B_r(p)$ denote the (four-dimensional) ball of radius $r$ centered at some point $p \in W^4$ and $\mu_{B_r(p)}$ its volume form.
A classic result by Yau states \cite[Theorem 3 \& Appendix (ii)]{Yau1976}:
If $\liminf_{r\to \infty} r^{-1} \int_{B_r(p)} \abs{\omega} \mu_{B_r(p)} = 0$, then $\int_{W^4} d \omega = 0$.

Near the boundary we can rely on \autoref{lem:vanishings-expansion-of-nahm-pole-solutions}.
Applied to the Kapustin-Witten solution $(\tilde A,\phi)$ with $\beta = \pi/2 - \theta = 0$, this yields
\begin{align}
	\omega = y^{-2} \Tr(\mathfrak{t}_1 \mathfrak{t}_2 + \mathfrak{t}_1 \mathfrak{t}_3) + y^{0} \Tr( \mathfrak{t}_1 (C^{1,+})_2 + \mathfrak{t}_1 (C^{1,+})_3 + (C^{1,+})_1 \mathfrak{t}_2 + (C^{1,+})_1 \mathfrak{t}_3) + \mathcal{O}(y^\delta) \qquad (y\to 0)
\end{align}
The term proportional to $y^{-2}$ vanishes and the term at constant order is assumed to be integrable on $X^3$, such that contributions from the Nahm pole to the integral $\int_{B_r(p)} \abs{\omega}$ are harmless.

At a non-compact end $[0,1)_s \times X^3$ of $W^4$, the asymptotic boundary condition states that $A + iB$ approaches a flat $SL(2,\C)$-connection on $X^3$ that satisfies $J^- \sigma(B,B) = 0$.
Equivalently, the component $\phi_1$ of $J^+ B$ commutes with the components $\phi_{2/3}$ of $J^- B$.
This is for example the case if $X^3$ is a product $S^1 \times \Sigma$, where $\Sigma$ is a Riemann surface and the almost complex structure $J$ is the direct sum of complex structures on the cylinder $[0,1)_s \times S^1$ and $\Sigma$.
In any case, since $[\phi_1,\phi_2] = [\phi_1,\phi_3] = 0$ we find that $\omega = \Tr(\phi_1 \phi_2) + \Tr(\phi_1\phi_3) \to 0$.
Assuming $A+ iB$ converges to the flat connection faster than the volume of the geodesic balls $B_r(p)$ grows, we conclude that $\int_{W^4} d\omega = 0$.

\paragraph{Conclusion:}
Since all boundary contributions to the exact term vanish in the limit $\epsilon \to 0$, we arrive at
\begin{align}
	\int_{M^5} \norm{\HW[v](A,B)}^2 = \int_{M^5} \norm{\dHW[v,J](A,B)}^2 + \lim_{\epsilon \to 0} \int_{M^5_\epsilon} d \chi = \int_{M^5} \norm{\dHW[v,J](A,B)}^2
\end{align}
Seeing that the integrands on both sides are non-negative, we find that whenever $\HW(A,B) = 0$ also $\dHW(A,B) = 0$, which concludes the proof.

\end{proof}